\documentclass[11pt]{article}

\usepackage{microtype}
\usepackage{amssymb}
\usepackage{graphicx}
\usepackage{subfigure}
\usepackage[usenames]{color}
\usepackage{colortbl}

\usepackage{enumitem}

\usepackage{latexsym}
\usepackage{fullpage, color}
\usepackage{array}
\usepackage{multirow}
\usepackage{tabularx}
\usepackage{wrapfig}
\usepackage{tablefootnote}
\usepackage[dvipsnames]{xcolor}
\usepackage{authblk}
\usepackage{algorithm,algorithmic}

\usepackage[utf8]{inputenc} 
\usepackage[T1]{fontenc}    
\usepackage{hyperref}       
\usepackage{url}          
\usepackage{booktabs}       
\usepackage{amsfonts}       
\usepackage{nicefrac}       
\usepackage{microtype} 
\usepackage{amsmath}
\usepackage{amsthm}
\usepackage{adjustbox}

\usepackage[capitalize,noabbrev]{cleveref}

 \newtheorem{theorem}{Theorem}[section]

 \newtheorem{lemma}[theorem]{Lemma}
\newtheorem{assumption}[theorem]{Assumption}
\newtheorem{example}[theorem]{Example}
 \newtheorem{definition}[theorem]{Defintion}

\DeclareMathOperator*{\argmin}{arg\,min}

\newcommand{\R}{{\mathbb R}}

\newcommand{\e}{\varepsilon}
\newcommand{\la}{\langle}
\newcommand{\ra}{\rangle}

\usepackage[textsize=tiny]{todonotes}

\title{Hyperfast Second-Order Local Solvers for Efficient
Statistically Preconditioned Distributed Optimization}

\author[1]{Pavel Dvurechensky} 
\author[2,8]{Dmitry Kamzolov} 
\author[3]{Aleksandr Lukashevich} 
\author[4]{Soomin Lee} 
\author[4]{Erik Ordentlich} 
\author[5]{C\'esar A. Uribe} 
\author[2,6,7]{Alexander Gasnikov} 

\affil[1]{Weierstrass Institute for Applied Analysis and Stochastics, Berlin, Germany}
\affil[2]{Moscow Institute of Physics and Technology, Dolgoprudny, Russia}
\affil[3]{Center for Energy Science and Technology, Skolkovo Institute of Science and Technology, Moscow, Russia}
\affil[4]{Yahoo! Research, Sunnyvale, CA}
\affil[5]{Rice University, Houston, TX}
\affil[6]{Institute for Information Transmission Problems RAS, Moscow, Russia}
\affil[7]{National Research University Higher School of Economics, Moscow, Russian Federation}
\affil[8]{{Mohamed bin Zayed University of Artificial Intelligence, Masdar City, Abu Dhabi, UAE}}

\begin{document}
\maketitle

\begin{abstract}
Statistical preconditioning enables fast methods for distributed large-scale empirical risk minimization problems.
In this approach, multiple worker nodes compute gradients in parallel, which are then used by the central node to update the parameter by solving an auxiliary (preconditioned) smaller-scale optimization problem.
The recently proposed Statistically Preconditioned Accelerated Gradient (SPAG) method~\cite{pmlr-v119-hendrikx20a} has complexity bounds superior to other such algorithms but requires an exact solution for computationally intensive auxiliary optimization problems at every iteration. In this paper, we propose an Inexact SPAG (InSPAG) and explicitly characterize the accuracy by which the corresponding auxiliary subproblem needs to be solved to guarantee the same convergence rate as the exact method. We build our results by first developing an inexact adaptive accelerated Bregman proximal gradient method for general optimization problems under relative smoothness and strong convexity assumptions, which may be of independent interest. Moreover, we explore the properties of the auxiliary problem in the InSPAG algorithm assuming Lipschitz third-order derivatives and strong convexity. For such problem class, we develop a linearly convergent Hyperfast second-order method and estimate the total complexity of the InSPAG method with hyperfast auxiliary problem solver. Finally, we illustrate the proposed method's practical efficiency by performing large-scale numerical experiments on logistic regression models. To the best of our knowledge, these are the first empirical results on implementing high-order methods on large-scale problems, we work with data where the dimension is of the order of~$3$ million, and the number of samples is~$700$ million.
\end{abstract}

\section{Introduction}

The efficient parallelization of large-scale learning is one of the most challenging problems in modern machine learning. Among several approaches, distributed computation and preconditioning have been shown effective in accelerating optimization algorithms, especially with increasing amounts of data~\cite{pmlr-v32-shamir14,pmlr-v119-hendrikx20a,JMLR:v21:19-764}. In this paper, we propose an efficient distributed optimization algorithm for solving the empirical risk minimization (ERM) problem:
\begin{equation}\label{eq:ERM1}
\min_{x \in \mathbb{R}^d} \left\{ f(x) \triangleq F(x) + h(x) \right\},
\end{equation}
where~$h(x)$ is a convex regularizer and~$F(x)$ is the empirical loss
\begin{equation}\label{eq:ERM1_1}
F(x) \triangleq \frac{1}{N}\sum_{i=1}^{N} \ell(x; \zeta_i).
\end{equation}
Here~$\mathcal{D}\triangleq \{\mathbf{\zeta}_i = (\mathbf{\xi}_i,\eta_i)\}_{i=1}^N$ is a set of~$N$ training data samples, and~$\ell$ is a convex loss function with respect to~$x$.  We assume that~$F$ is~$L_F$-smooth and~$\mu_F$-strongly convex, i.e.,
\begin{equation}
\label{eq:rel_str_conv_simple}
\mu_FI_d \preceq \nabla^2 F(x) \preceq L_F I_d, 
\end{equation}
where~$I_d$ is the~$d$-dimensional identity matrix. The condition number of~$F$ is denoted as~$\kappa_F = L_F/\mu_F$, and the solution to~\eqref{eq:ERM1} is denoted as~$x_*$.

Sum-type optimization problems of the form~\eqref{eq:ERM1} are used to model various statistical learning problems, including least squares regression, logistic regression, and support vector machines. One characteristic of modern uses of models like~\eqref{eq:ERM1} is the so-called large-scale regime, i.e., when~$N$ is very large. Large~$N$ poses additional challenges related to the storage and processing of data, which in turn drives the need for modern distributed/federated architectures~\cite{wang2018giant} that take advantage of parallel processing capabilities~\cite{hendrikx2020optimal}, e.g., Apache Spark~\cite{yang2013trading}, Parameter Server~\cite{li2014scaling} and MapReduce~\cite{dean2008mapreduce}. 

In practice, when~$N$ is very large, the complete set of data points~$\mathcal{D}$ cannot be stored or is not accessible at a single machine. Instead, data is distributed among~$m$ computing units/nodes/agents such that~$\mathcal{D} = \{\mathcal{D}_1, \ldots, \mathcal{D}_m\}$. Moreover, the distributed setup assumes there is a central node, that is able to communicate with all the worker nodes. Without loss of generality we assume that~$N = mn$, i.e., machine~$j\in \{1,\hdots,m\}$ locally stores~$n$ samples~$\mathcal{D}_j = \{\mathbf{\xi}_i^{(j)}, \eta_i^{(j)}\}_{i=1}^n$. Specifically, each agent~$j$ has a local empirical risk, denoted as~$F_j (x) \triangleq ({1}/{n})\sum_{i=1}^n \ell(x; \mathbf{\xi}_i^{(j)}, \eta_i^{(j)})$. Thus,
\begin{equation}\label{eq:ERM2}
F(x) {=} \frac{1}{m}\sum_{j=1}^{m} F_j(x) 
 {=} \frac{1}{nm}\sum_{j=1}^{m} \sum_{i=1}^{n} \ell(x; \mathbf{\xi}_i^{(j)},\eta_i^{(j)}) .
\end{equation}

The centralized distributed optimization architecture described above, with a central node and a number of workers, typically involves of two resources: communication and computation. Communication is usually regarded as the most valuable resource~\cite{lan17}. Thus, recent efforts~\cite{pmlr-v32-shamir14,pmlr-v119-hendrikx20a,JMLR:v21:19-764} have been focused on the efficiency of communications, where one seeks to minimize~\eqref{eq:ERM2} with a minimal number of communication rounds between the workers and the central node.

\textit{Recent Distributed Optimization Approaches: }
The distributed approximate Newton-type method (DANE)~\cite{pmlr-v32-shamir14} has been one of the most popular second-order methods for communication-efficient distributed machine learning. DANE improves the polynomial dependency of the iteration complexity on the condition number~$\kappa_F$ of first-order methods for distributed empirical risk minimization problems, compared to the geometric rates available for centralized, i.e., non-distributed, methods~\cite{nesterov2018lectures}. Particularly, DANE has an iteration (communication) complexity of~$\widetilde O (\kappa_F^2/n)$\footnote{The~$\widetilde O$-notation means non-asymptotic inequality up to constant and poly-logarithmic factors. More precisely,~$A = \widetilde{O}(B)$ if there exist constants~$C,a>0$ such that~$A\leq C B\ln^a\frac{1}{\varepsilon}$.} for quadratic functions, and~$\widetilde O (\kappa_F)$ for convex non-quadratic functions. However, DANE requires the exact solution of a carefully constructed subproblem, which can be impractical~\cite{pmlr-v32-shamir14}. An inexact version of DANE, termed \mbox{InexactDANE}~\cite{reddi2016aide}, and its accelerated variant, termed AIDE~\cite{reddi2016aide}, achieve an iteration complexity of~$\widetilde O (\kappa_F)$, and~$\widetilde O (\sqrt{\kappa_F})$ respectively, without requiring exact solutions of the auxiliary subproblem. For quadratic functions InexactDANE and AIDE have an iteration complexity of~$\widetilde O (\kappa_F^2/n)$ and~$\widetilde O (\sqrt{\kappa_F}/n^{1/4})$ respectively. Nevertheless, the advantage of preconditioning, where the condition number is effectively reduced as~$n$ increases, was only shown for quadratic problems. Recently, in~\cite{JMLR:v21:19-764}, the authors showed that the preconditioning effect holds locally for a variation of DANE termed DANE-HB with inexact solutions to the local subproblem. Specifically, an iteration complexity of~$\widetilde O (d^{1/4}\sqrt{\kappa_F}/n^{1/4})$ was shown to hold in a neighborhood around the optimal point for non-quadratic convex functions. Additionally, for linear prediction models, an improved global bound of~$\widetilde O (\sqrt{\kappa_F}/n^{1/4})$  was shown~\cite{JMLR:v21:19-764} to be achieved by the D$^2$ANE Algorithm. 
In~\cite{pmlr-v37-zhangb15} the authors propose the DiSCO algorithm with global bounds ~$\widetilde O (\sqrt{\kappa_F}/n^{1/4})$ for quadratic functions and~$\widetilde O (d^{1/4}\sqrt{\kappa_F}/n^{1/4})$ for self-concordant functions which are a different class than functions satisfying \eqref{eq:rel_str_conv_simple}.
One of the main observations in~\cite{JMLR:v21:19-764} is that the looseness in the bounds of DANE and AIDE came from the reduce (model aggregation) step done by the central node. Thus, DANE-HB and D$^2$ANE build their results from a modified structure. The worker nodes compute gradients and communicate them back to the central node, which solves the preconditioned auxiliary subproblem. Such algorithmic structure was used in~\cite{pmlr-v119-hendrikx20a} recently, where the authors proposed the Statistically Preconditioned Accelerated Gradient (SPAG) method. SPAG has an iteration complexity of~$\widetilde O (\sqrt{\kappa_F}/n^{1/4})$ for quadratic functions with direct acceleration, instead of using the Catalyst framework~\cite{catalyst}. SPAG was also shown to have an asymptotic iteration complexity of~$\widetilde O (\sqrt{\kappa_F}/n^{1/4})$, with empirical evidence that such rate behavior holds non-asymptotically in practice. However, exact solvers for the auxiliary subproblem on the central node are required. Such convergence rates match complexity lower bounds~\cite{dragomir2019optimal,arjevani}. In a more challenging setup (which we do not consider in this paper) of decentralized distributed optimization~\cite{sun2019distributed} propose an algorithm with iteration complexity~$\widetilde O ( \kappa_F/\sqrt{n})$ and similar up to a network-dependent factor communication complexity.

Although SPAG obtains the near-optimal iteration complexity for distributed algorithms applied to~\eqref{eq:ERM1}, and \eqref{eq:ERM2}, it strongly depends on the ability to exactly solve an intermediate auxiliary optimization subproblem (usually in the form of a non-Euclidean Bregman projection), whose complexity was not explicitly taken into account in~\cite{pmlr-v119-hendrikx20a}. More importantly, as pointed out in~\cite{pmlr-v119-hendrikx20a}, such an intermediate problem is computationally hard, and the accuracy of its solution dramatically affects the performance of the whole method. \textit{We solve this issue in this paper.} 

Our solution's key innovation is explicitly considering the auxiliary subproblem's inexactness and quantifying how it affects the convergence rate of the whole algorithm. Moreover, for the case of functions with high-order bounded derivatives (e.g., logistic regression or softmax problems~\cite{pmlr-v125-bullins20a}), we provide a Hyperfast second-order method that efficiently computes the approximate solution of the subproblem. This approach builds upon the line of works on \textit{implementable} tensor methods for \textit{convex} problems recently initiated\footnote{
We underline that the main words here are \textit{implementable} and \textit{convex}. Adaptive tensor methods with optimal complexity guarantees for non-convex problems were proposed earlier in~\cite{birgin2017worst-case,carmon2020lower,cartis2019universal}, and previously known tensor methods for convex problems~\cite{baes2009estimate} did not necessarily have convex auxiliary problem in each iteration.
}
by Yu. Nesterov~\cite{nesterov2019implementable}, where it was shown that the third-order method for convex problems with Lipschitz third-order derivative could have a convex subproblem and its solution can be efficiently implemented. 
Later,~\cite{gasnikov2019near} proposed near-optimal tensor methods with complexity bounds which match up to a logarithmic factor the lower bounds for highly-smooth convex optimization.~\cite{nesterov2020superfast} proposes a third-order tensor method with third-order derivative approximated by finite-difference of gradients, which leads to a Superfast second-order method with convergence rate~$O(1/k^4)$ for convex functions with Lipschitz third-order derivative. As a next step,~\cite{nesterov2020inexact} proposes an inexact accelerated high-order proximal point method which allows improving, compared to Superfast second-order method, the convergence rate to~$O(1/k^5)$ up to logarithmic factors. In parallel to the previous work and inspired by~\cite{nesterov2020superfast}, the authors of~\cite{kamzolov2020near} proposed a Hyperfast second-order method with the same convergence rate, but based on another accelerated high-order method developed in~\cite{gasnikov2019near}. In this paper, we extend both methods to the setting of strongly convex minimization problems and apply them to solve the intermediate auxiliary optimization subproblem in each iteration of our inexact version of SPAG.

\paragraph{Contributions}
SPAG is one of the fastest distributed methods (in terms of communication steps number) for the minimization of~\eqref{eq:ERM1}, and \eqref{eq:ERM2} with i.i.d. samples~\cite{pmlr-v119-hendrikx20a}. Moreover, the Hyperfast second-order method is the best known (near-optimal) second-order method to minimize convex functions with Lipschitz third-order derivatives. We argue that the extended combination of the proposed inexact SPAG and the new Hyperfast second-order method provides a useful approach to construct new efficient distributed algorithms. Specifically, in SPAG, the central node solves a problem with a similar structure as~\eqref{eq:ERM1}, but with a smaller number~$n$ of data samples. Therefore, with a reduced number of samples, the complexity of calculating the Hessian is comparable (due to the sum-type structure of~$F$) with its inversion by the matrix inversion lemma~\cite{cormen2009introduction} and modern practical versions of Strassen-type algorithm~\cite{huang2016strassen}. In this regime, at the central node, Hyperfast second-order methods outperform existing variance-reduced stochastic first-order schemes.
\textit{We extend the theoretical analysis of inexact statistical preconditioning methods alongside high-order methods and show that they jointly provide an efficient second-order method that outperforms (from theoretical and practical points of view) well-known (stochastic) first-order schemes.} 

The main contributions of this paper are as follows:
\begin{itemize}[leftmargin=5mm, noitemsep, topsep = .2mm]
    \item Since SPAG is based on the accelerated Bregman proximal gradient method for relatively smooth and strongly-convex problems, we first propose an inexact accelerated Bregman proximal gradient method for general convex optimization problems. Our algorithm is based on an inexact model for the objective, which subsumes the setting of relatively smooth and (strongly-)convex problems and the setting of inexact first-order oracles. Our algorithm also allows for approximate Bregman projections. We estimate the convergence rate and rates of inexactnesses accumulation.
    \item We propose an Inexact Statistically Preconditioned Accelerated Gradient (InSPAG) method for distributed optimization problem \eqref{eq:ERM1}, \eqref{eq:ERM2}, and explicitly characterize the accuracy by which the corresponding auxiliary subproblem needs to be solved to guarantee the same convergence rate as the exact method, i.e.,~$\widetilde O (\sqrt{\kappa_F}/n^{1/4})$. Our method is not a direct extension and has a slightly simpler structure than the method in~\cite{pmlr-v119-hendrikx20a}.
    \item We extend and generalize the Hyperfast second-order method~\cite{nesterov2020inexact,kamzolov2020near}, recently proposed for smooth and convex problems, to the class of uniformly, and especially strongly, convex functions. We show a linear convergence rate for this problem class. 
    \item Based on an example of sparse logistic regression, we discuss the distributed optimization problem regime, for which Hyperfast second-order optimization methods provide a theoretical advantage over classical first-order methods for the problem size, dimension, and desired accuracy of the solution.
    \item We provide experimental results in application to large-scale machine learning problems that show the efficiency of the use of high-order methods in practice. To the authors' best knowledge, this is one of the first attempts to apply near-optimal tensor methods for real data and applications. Specifically, we test the proposed algorithm on a proprietary data set with~$710$ million entries and a dimension of~$3.2$ million.  
\end{itemize}

\paragraph{Outline}  In Section \ref{S:AGM}, we introduce the inexact accelerated Bregman proximal gradient method for general convex optimization problems. This includes defining the concept of the inexact model of the objective, illustrating it by examples, presenting the algorithm and its convergence rate theorem together with its proof. Section~\ref{sec:problem} presents the setting for statistically preconditioned distributed algorithms, introduces InSPAG algorithm and its convergence rate theorem. After that, we present the Hyperfast second-order method for the auxiliary subproblem of the InSPAG, estimate its complexity and combine the building blocks to obtain the total complexity of the whole approach. We finish this section by discussing the regime in which our approach is superior to applying stochastic variance-reduced algorithms. Section~\ref{sec:numerics} presents our experimental results. For the sake of completeness in Section \ref{app:uni} we present Hyperfast second-order method for uniformly convex functions. We finalize with conclusions in Section~\ref{sec:conclusions}.

\section{Accelerated Gradient Method under Inexactness and Relative Smoothness}
\label{S:AGM}

In this section, we propose a general accelerated first-order algorithm that will be used in the next section to propose our InSPAG method
for distributed optimization. We believe that the results of this section may be of independent interest. This section is, to an extent, independent of the others sections and the reader interested in the distributed optimization may skip this section since in what follows only the main result of this section (Theorem \ref{Th:fast_str_conv_adap}) will be used. We consider the following general optimization problem
\begin{equation}
\label{eq:Problem}
\min_{x\in Q} f(x),
\end{equation}
where~$Q$ is a convex subset of finite-dimensional vector space~$E$. Our goal is to develop a general accelerated inexact gradient method capable to work under relative smoothness and strong convexity assumptions~\cite{bauschke2016descent,lu2018relatively}.
We consider two types of inexactness: inexact information on the objective function and inexact generalized projection. 

Before we give more details, we introduce some general notations. Let~$E$ be an~$d$-dimensional real vector space and~$E^*$ be its dual. We denote the value of a linear function~$g \in E^*$ at~$x\in E$ by~$\la g, x \ra$. Let~$\|\cdot\|$ be some norm on~$E$,~$\|\cdot\|_{*}$ be its dual, defined by~$\|g\|_{*} = \max\limits_{x} \big\{ \la g, x \ra, \| x \| \leq 1 \big\}$. 
Let~$\phi$ be a convex function on~$Q$, which is continuously differentiable on the relative interior~${\rm ri}Q$ of~$Q$. Let~$D_{\phi}[y](x) = \phi(x) - \phi(y) - \la \nabla \phi(y), x - y \ra$,~$x \in Q, y \in {\rm ri}Q$ be the corresponding Bregman divergence. Based on the Bregman divergence we introduce the following two definitions of inexactness.

\begin{definition}[Inexact model~\cite{stonyakin2021inexact}]\label{Def:RelStronglyConvex}
Let~$\delta,L,\mu,m \geq 0$. We say that~$(f_{\delta}(y),\psi_{\delta}(x, y))$ is a~$(\delta, L, \mu, m, \phi)$-model of the function~$f$ at a given point~$y$ iff, for all~$x \in Q$,
\begin{equation}
\label{eq:Str_Conv_Model}
{\mu D_{\phi}[y](x)} \leq {f(x) - \left(f_{\delta}(y) + \psi_{\delta}(x, y)\right)} \leq {LD_{\phi}[y](x) + \delta},
\end{equation}
~$\psi_{\delta}(x, y)$  is convex in~$x$, satisfies~$\psi_{\delta}(x,x)=0$ for all~$x \in Q$ and \begin{equation}\label{mstongmodel}
\psi(x) \geqslant \psi(z) + \la g, x- z \ra + m D_{\phi}[z](x), \quad \forall x, z \in Q, \,\, \forall g \in \partial \psi(z),
\end{equation}
where for fixed~$y \in Q$ and any~$x \in Q$ we denote~$\psi(x) = \psi_{\delta}(x, y)$.  
\end{definition}

\begin{definition}[Inexact generalized projection~\cite{ben-tal2015lectures}]
\label{Def:solNemirovskiy}
For a convex optimization problem
$\min_{x \in Q} \Psi(x)$
 and~$\widetilde{\delta} \geq 0$, we denote by~$\text{\rm Arg}\min_{x \in Q}^{\widetilde{\delta}}\Psi(x)$~a set of points ~$\widetilde{x}$ such  that
\begin{gather}\label{eqv_inex_sol}
 \exists h \in \partial\Psi(\widetilde{x}) :\forall x \in Q \,\,  \to\, \langle h, x - \widetilde{x} \rangle \geq -\widetilde{\delta}.
\end{gather}
We denote by~$\argmin_{x \in Q}^{\widetilde{\delta}}\Psi(x)$  some element of~$\text{\rm Arg}\min_{x \in Q}^{\widetilde{\delta}}\Psi(x)$.
\end{definition}

Optimization algorithms with inexact model of the objective were extensively studied in~\cite{stonyakin2021inexact} and are generalizations of first-order algorithms with inexact oracle~\cite{devolder2014first,dvurechensky2016stochastic}. We now give two particular examples that are covered by the inexact model framework and refer to~\cite{stonyakin2021inexact} for further examples.

\begin{example}{\bf Relative smoothness and relative strong convexity,~\cite{bauschke2016descent, lu2018relatively}.}
\label{Ex:rel_smooth}
Assume that ~$\phi(x)$ is differentiable, and in \eqref{eq:Problem}, the objective~$f$ is differentiable, relatively smooth~\cite{bauschke2016descent,lu2018relatively} and  strongly convex ~\cite{lu2018relatively} relative to~$\phi$, i.e., for some~$\mu \geq 0$ and~$L >0$,
\[
 \mu D_{\phi}[y](x) \leq f(x)-f(y) -\la \nabla f(y), x-y \ra  \leq  LD_{\phi}[y](x), \; \forall x,y \in Q.
\]
Then, clearly, Definition \ref{Def:RelStronglyConvex} holds with~$m = 0$,~$\delta = 0$,~$f_{\delta}(y) = f(y)$,~$\psi_\delta(x,y) = \langle \nabla f(y), x - y \rangle$. Importantly, the function~$\phi$ is not required to be strongly convex. Our InSPAG relies on this particular example.
\end{example}

\begin{example}{\bf Composite optimization,~\cite{beck2009fast, nesterov2013gradient}.}
\label{Ex:composite}
Assume that in \eqref{eq:Problem},~$f(x) = g(x) + h(x)$ with convex ~$L$-smooth w.r.t. norm~$\|\cdot\|$ term~$g(x)$ and simple convex term~$h(x)$ which is usually called composite. In this case we assume that~$\phi(x)$ is 1-strongly-convex w.r.t~$\|\cdot\|$, and define~$f_{\delta}(y) =  g(y) + h(y)$ and~$\psi_\delta(x,y) = \langle \nabla g(y), x - y \rangle + h(x) - h(y)$. 
Then, clearly,
\begin{gather*}
f(x) - \left(f_{\delta}(y) + \psi_{\delta}(x, y)\right) = 
g(x) - (g(y) + \langle \nabla g(y), x - y \rangle ).
\end{gather*}
By convexity of~$g$, we have~$0 \leq g(x) - (g(y) + \langle \nabla g(y), x - y \rangle )$. At the same time, by the~$L$-smoothness of~$g$ and 1-strong-convexity of~$\phi(x)$, 
\[
g(x) - (g(y) + \langle \nabla g(y), x - y \rangle )\leq \frac{L}{2}\|x-y\|^2 \leq LD_{\phi}[y](x).
\]
From the combination of the above two relations, it is clear that \eqref{eq:Str_Conv_Model} holds with~$\delta=0$ and~$\mu=0$ and we are in the situation of Definition \ref{Def:RelStronglyConvex} with~$m=0$ since~$\psi_\delta(x,y)$ is convex in~$x$. 
\end{example}

In~\cite{stonyakin2021inexact}, to develop an accelerated algorithm, the authors use a different assumption where in the r.h.s. of \eqref{eq:Str_Conv_Model} the Bregman divergence~$D_{\phi}[y](x)$ is substituted with~$\frac{1}{2}\|x-y\|^2$, and assume that~$\phi$ is 1-strongly-convex w.r.t.~$\|\cdot\|$. This, unfortunately, restricts the range of applications of the algorithm, and we use a weaker set of assumptions in Definition \ref{Def:RelStronglyConvex}. At the same time,~\cite{dragomir2019optimal} showed that it is not possible to develop an accelerated algorithm in the relative smoothness setting without additional assumptions. Thus, we introduce the following assumption on the Bregman divergence~$D_{\phi}[y](x)$ and note that the range of applications is still wider than for the approach of~\cite{stonyakin2021inexact}. We also note that this assumption is simpler than the one in~\cite{pmlr-v119-hendrikx20a} and is a version of triangle scaling gain introduced in~\cite{hanzely2021accelerated} and triangle lower bound property of~\cite{florea2019exact}.
\begin{assumption}\label{assum:Bregman}
There exists a constant~$G \geq 1$ such that for all~$x,y,u,u_{+} \in {\rm ri \; dom} \phi$ such that~$x-y=\tau(u_+-u)$ for some~$\tau \in [0,1]$ it holds that
\begin{equation}
    \label{eq:D_phi_assum}
    D_{\phi}[y](x) \leq G \tau^2 D_{\phi}[u](u_+).
\end{equation}
\end{assumption}
This assumption can be seen as a relaxation of homogeneity of degree 2. The simplest example when this property holds is when~$D_{\phi}[y](x) = \frac{1}{2}\|y-x\|^2$. In this case~$G=1$. We also note that our algorithm is adaptive to constant~$G$ which means that the property \eqref{eq:D_phi_assum}  is sufficient to hold only locally.

The proposed accelerated gradient method with inexact model is listed below as Algorithm \ref{alg:FastAlg2_strong}. Unlike~\cite{pmlr-v119-hendrikx20a,hanzely2021accelerated,florea2019exact}, our algorithm is simultaneously adaptive to the ``Lipschitz'' constant~$L$ (see Definition \ref{Def:RelStronglyConvex}) and constant~$G$ in Assumption \ref{assum:Bregman}, which is expressed in constant~$M$ that plays the role of the product~$LG$. Also, unlike~\cite{pmlr-v119-hendrikx20a,hanzely2021accelerated,florea2019exact}, our algorithm allows two types of inexactness covered by Definitions \ref{Def:RelStronglyConvex} and \ref{Def:solNemirovskiy}. Finally, unlike~\cite{hanzely2021accelerated,florea2019exact}, our algorithm has linear convergence when~$\mu>0$. We also note that we allow the accuracies~$\delta,\widetilde{\delta}$ in Definition \ref{Def:RelStronglyConvex} and \ref{Def:solNemirovskiy} to depend on the iteration counter~$k$, which is expressed by the sequences~$\{\delta_k,\widetilde{\delta}_k\}_{k \geq 0}$.

\begin{algorithm}[h!]
\caption{{Accelerated gradient method with~$(\delta, L,
\mu, m, \phi)$-model}}
\label{alg:FastAlg2_strong}
\begin{algorithmic}[1]
\STATE \textbf{Input:}~$x_0$ is the starting point,
$\mu \geq 0$,~$m \geq 0$,
$\{\delta_k\}_{k\geq 0}$ and
\STATE Set
$y_0 := x_0$,~$u_0 := x_0$,~$\alpha_0 := 0$,~$A_0 := \alpha_0$
\FOR{$k \geq 0$}
\STATE Find the smallest integer~$i_k \geq 0$ such that
\vspace{-0.25cm}
\begin{equation}
\begin{gathered}
f_{\delta_k}(x_{k+1}) \leq f_{\delta_k}(y_{k+1}) + \psi_{\delta_k}(x_{k+1}, y_{k+1}) +\frac{M_{k+1}\alpha^2_{k+1}}{A_{k+1}^2}D_{\phi}[u_{k}](u_{k+1}) + \delta_k,
\label{exitLDL_strong}
\end{gathered}
\end{equation}
where~$M_{k+1} = 2^{i_k-1}M_k$,~$\alpha_{k+1}$ is the largest root  of the equation
\begin{gather}
\label{alpha_def_strong}
A_{k+1}{(1 + A_k \mu + A_k m)}=M_{k+1}\alpha^2_{k+1},\quad A_{k+1} := A_k + \alpha_{k+1}, \;\text{and}
\end{gather}
\vspace{-0.7cm}
\begin{gather}
y_{k+1} := \frac{\alpha_{k+1}u_k + A_k x_k}{A_{k+1}}, \label{eqymir2DL_strong}
\end{gather}
\vspace{-0.7cm}
\begin{equation*}
\hspace{-2em}\Phi_{k+1}(x):=\alpha_{k+1}\psi_{\delta_k}(x, y_{k+1}) + {(1 + A_k(\mu + m))} D_{\phi}[u_k](x) + {\alpha_{k+1} \mu D_{\phi}[y_{k+1}](x)},
\end{equation*}
\vspace{-0.7cm}
\begin{equation}\label{equmir2DL_strong}
u_{k+1} := {\argmin_{x \in Q}}^{\widetilde{\delta}_k}\Phi_{k+1}(x), \;\;\text{for some~$\widetilde{\delta}_k \geq 0$}
\end{equation}
\vspace{-0.7cm}
\begin{gather}
x_{k+1} := \frac{\alpha_{k+1}u_{k+1} + A_k x_k}{A_{k+1}}. \label{eqxmir2DL_strong}
\end{gather}
\vspace{-1cm}
\STATE Set~$k:=k+1$.
\ENDFOR
\STATE \textbf{Ouput:}~$x_k$
\end{algorithmic}
\end{algorithm}

The following is the convergence rate result for the proposed algorithm.
\begin{theorem}
\label{Th:fast_str_conv_adap}
Assume that ~$(f_{\delta}(y),\psi_{\delta}(x, y))$ is a~$(\delta, L, \mu, m, \phi)$-model according to Definition \ref{Def:RelStronglyConvex}. Also assume that~$D_{\phi}[y](x)$ satisfies Assumption \ref{assum:Bregman}. Then, after~$N$ iterations of Algorithm \ref{alg:FastAlg2_strong}, we have
\begin{align}
    &f(x_N) - f(x_*) \leq \frac{D_{\phi}[u_0](x_*)}{A_N} + \frac{2\sum_{k=0}^{N-1}A_{k+1}\delta_k}{A_N} + \frac{\sum_{k=0}^{N-1}\widetilde{\delta}_k}{A_N}, \label{Th:fast_str_conv_adap:result_1} \\
    &D_{\phi}[u_N](x_*)\leq \frac{D_{\phi}[u_0](x_*)}{(1 + A_N\mu + A_Nm)} + \frac{2\sum_{k=0}^{N-1}A_{k+1}\delta_k}{(1 + A_N\mu + A_Nm)} + \frac{\sum_{k=0}^{N-1}\widetilde{\delta}_k}{(1 + A_N\mu +  A_Nm)}. \label{Th:fast_str_conv_adap:result_2}
    \end{align}
\end{theorem}
In order to prove Theorem \ref{Th:fast_str_conv_adap} we need the following technical Lemma.
\begin{lemma}[\cite{stonyakin2021inexact},\, Lemma 3.5.]
    \label{lemma:fast_str_conv}
	Let~$\psi(x)$ be a relatively~$m$-strongly convex function relative to~$\phi$ with~$m \geq 0$, i.e. \eqref{mstongmodel} holds, and
	\begin{gather*}
	y = {\argmin_{x \in Q}}^{\widetilde{\delta}} \{\psi(x) + \beta D_{\phi}[z](x) + \gamma D_{\phi}[u](x)\},
	\end{gather*}
	where~$\beta \geq 0$ and~$\gamma \geq 0$.
Then, for all~$x \in Q$,
	\begin{equation*}
	\psi(x) + \beta D_{\phi}[z](x) + \gamma D_{\phi}[u](x) \geq \psi(y) + \beta V[z](y) + \gamma D_{\phi}[u](y) + (\beta + \gamma + m)D_{\phi}[y](x) - \widetilde{\delta}.
	\end{equation*}
\end{lemma}

\begin{proof}[Proof of Theorem \ref{Th:fast_str_conv_adap}]
We start by proving the correctness of the algorithm, i.e. that if we fix iteration~$k$, there exists~$i_k \geq 0$ such that  \eqref{exitLDL_strong} holds.
By Definition~\ref{Def:RelStronglyConvex} with~$ x = y$, we have~$f_{\delta_k}(y)  \leq f(y)$. Thus, from \eqref{eq:Str_Conv_Model}
\begin{equation}
    f_{\delta_k}(x_{k+1}) \leq  
    f_{\delta_k}(y_{k+1}) + \psi_{\delta_k}(x_{k+1},y_{k+1})  +
    LD_{\phi}[y_{k+1}](x_{k+1}) + \delta_k.
\end{equation}
Combining this with Assumption \ref{assum:Bregman} and using \eqref{eqymir2DL_strong}, \eqref{eqxmir2DL_strong}, we further obtain
\begin{equation}
    f_{\delta_k}(x_{k+1}) \leq  
    f_{\delta_k}(y_{k+1}) + \psi_{\delta_k}(x_{k+1},y_{k+1})  +
    \frac{LG\alpha_{k+1}^2}{A_{k+1}^2}D_{\phi}[u_{k}](u_{k+1}) + \delta_k.
\end{equation}
Since~$M_{k+1} = 2^{i_k-1}M_k$, we see that as soon as~$M_{k+1} \geq LG$, \eqref{exitLDL_strong} holds. Thus, the algorithm is correctly defined. Note also that by the same reason we have
\begin{equation}
    \label{eq:M_bound}
    M_{k+1} \leq 2LG.
\end{equation}

Our next goal is to prove that, for all~$x \in Q$, we have
\begin{align}
    &A_{k+1} f(x_{k+1}) - A_{k} f(x_{k}) + (1 + A_{k+1} \mu + A_{k+1} m) D_{\phi}[u_{k+1}](x) \notag \\
    &- (1 + A_k \mu + A_k m)D_{\phi}[u_{k}](x) 
    \leq \alpha_{k+1}f(x) + 2\delta_k A_{k+1} + \widetilde{\delta}_k. \label{eq:AGM_per_iter}
\end{align}

Since by Definition \ref{Def:RelStronglyConvex} with~$x = y$, we get~$f(x) - \delta_k \leq  f_{\delta_k}(x)  \leq f(x)$, and, using \eqref{exitLDL_strong}, we have
	\begin{align*}
	f(x_{k+1}) \stackrel{\eqref{eq:Str_Conv_Model}}{\leq} f_{\delta_k}(x_{k+1}) + \delta_k \stackrel{\eqref{exitLDL_strong}}{\leq} &f_{\delta_k}(y_{k+1}) + \psi_{\delta_k}(x_{k+1},y_{k+1})  \\
	&+
    \frac{M_{k+1}\alpha_{k+1}^2}{A_{k+1}^2}D_{\phi}[u_{k}](u_{k+1}) + 2\delta_k.
	\end{align*}
Substituting in this expression definition 	\eqref{eqxmir2DL_strong} of the point~$x_{k+1}$, using that~$A_{k+1}=A_k + \alpha_{k+1}$ and that, by Definition \ref{Def:RelStronglyConvex},~$\psi_{\delta_k}(\cdot,y)$ is convex, we have
\begin{align*}
	f(x_{k+1})
	&\leq\frac{A_k}{A_{k+1}}\left(f_{\delta_k}(y_{k+1}) + \psi_{\delta_k}(x_k, y_{k+1})\right)+\frac{\alpha_{k+1}}{A_{k+1}}\left(f_{\delta_k}(y_{k+1}) +
	 \psi_{\delta_k}(u_{k+1}, y_{k+1})\right)\\
	 &\hspace{2em}+ \frac{M_{k+1}\alpha_{k+1}^2}{A_{k+1}^2}D_{\phi}[u_{k}](u_{k+1}) + 2\delta_k.
\end{align*}	
In view of the definition \eqref{alpha_def_strong} of the sequence~$\alpha_{k+1}$ and left inequality in \eqref{eq:Str_Conv_Model}, we obtain
   \begin{align}
   \begin{split}
   \label{lemma:fast_str_conv:ineq_1}
	 f(x_{k+1}) &\leq \frac{A_k}{A_{k+1}}f(x_{k})+\frac{\alpha_{k+1}}{A_{k+1}}\Big(f_{\delta_k}(y_{k+1}) + \psi_{\delta_k}(u_{k+1},y_{k+1})\\
	 &\hspace{2em}+ \frac{1 + A_k \mu + A_k m}{\alpha_{k+1}}D_{\phi}[u_k](u_{k+1})\Big) + 2\delta_k.
	 \end{split}
  \end{align}
By Lemma \ref{lemma:fast_str_conv}, for the optimization problem in \eqref{equmir2DL_strong} with~$\psi(x) = \alpha_{k+1}\psi_{\delta_k}(x, y_{k+1})$,~$\beta = 1 + A_k\mu + A_k m$,~$z = u_k$,~$\gamma = \alpha_{k+1}\mu$, and~$u = y_{k+1}$,
it holds that
\begin{align*}
&\alpha_{k+1}\psi_{\delta_k}(u_{k+1}, y_{k+1}) + (1 + A_k \mu + A_k m)D_{\phi}[u_{k}](u_{k+1}) + \alpha_{k+1} \mu D_{\phi}[y_{k+1}](u_{k+1}) \\
&\hspace{2em}+ (1 + A_{k+1}\mu + A_{k+1} m)D_{\phi}[u_{k+1}](x) - \widetilde{\delta}_k \\
&\leq \alpha_{k+1}\psi_{\delta_k}(x, y_{k+1}) + (1 + A_k \mu + A_k m)D_{\phi}[u_{k}](x) + \alpha_{k+1} \mu D_{\phi}[y_{k+1}](x).
\end{align*}
From the fact that~$D_{\phi}[y_{k+1}](u_{k+1}) \geq 0$, we have
\begin{align}
\begin{split}
   \label{lemma:fast_str_conv:ineq_2}
&\alpha_{k+1}\psi_{\delta_k}(u_{k+1}, y_{k+1}) + (1 + A_k\mu + A_k m)D_{\phi}[u_{k}](u_{k+1})\\
&\leq \alpha_{k+1}\psi_{\delta_k}(x, y_{k+1}) + (1 + A_k\mu + A_k m)D_{\phi}[u_{k}](x)\\
&\hspace{2em}- (1 + A_{k+1}\mu +A_{k+1} m)D_{\phi}[u_{k+1}](x) + \alpha_{k+1} \mu D_{\phi}[y_{k+1}](x) + \widetilde{\delta}_k.
\end{split}
\end{align}
Combining \eqref{lemma:fast_str_conv:ineq_1} and \eqref{lemma:fast_str_conv:ineq_2}, we obtain
  \begin{align*}
  f(x_{k+1})
     &\leq \frac{A_k}{A_{k+1}} f(x_k)+\frac{\alpha_{k+1}}{A_{k+1}}\Big(f_{\delta_k}(y_{k+1}) + \psi_{\delta_k}(x,y_{k+1}) + {\mu D_{\phi}[y_{k+1}](x)}\\
	 &\hspace{2em}+ \frac{1 + A_k \mu + A_k m}{\alpha_{k+1}}D_{\phi}[u_k](x) \\
	 &\hspace{2em}- \frac{1 + A_{k+1} \mu + A_{k+1} m}{\alpha_{k+1}}D_{\phi}[u_{k+1}](x) + \frac{\widetilde{\delta}_k}{\alpha_{k+1}}\Big) + 2\delta_k.
\end{align*}
We finish the proof of \eqref{eq:AGM_per_iter} applying the left inequality in \eqref{eq:Str_Conv_Model}:
\begin{align*}
    f(x_{k+1})
	 &\leq \frac{A_k}{A_{k+1}} f(x_k) +
	 \frac{\alpha_{k+1}}{A_{k+1}} f(x) + \frac{1 + A_k \mu + A_k m}{A_{k+1}}D_{\phi}[u_k](x) \\
	 & \hspace{2em}- \frac{1 + A_{k+1} \mu + A_{k+1} m}{A_{k+1}}D_{\phi}[u_{k+1}](x) + 2\delta_k + \frac{\widetilde{\delta}_k}{A_{k+1}}.
\end{align*}

We now telescope the inequality \eqref{eq:AGM_per_iter} for~$k$ from~$0$ to~$N-1$ and take~$x = x_*$:
\begin{align}
    A_{N} f(x_{N})\leq & A_{N}f(x_*) + D_{\phi}[u_0](x_*) - (1 + A_N(\mu + m)) D_{\phi}[u_N](x_*)\notag  \\
    & \hspace{2em}+ 2\sum_{k=0}^{N-1}A_{k+1}\delta_k + \sum_{k=0}^{N-1}\widetilde{\delta}_k. \label{Th:str_conv_adap:proof:ineq_1}
\end{align}
Since~$V[u_{k+1}](x_*) \geq 0$ for all~$k \geq 0$, we have
\begin{align*}
	&A_{N} f(x_{N}) - A_{N}f(x_*) \leq D_{\phi}[u_0](x_*)  + 2\sum_{k=0}^{N-1}A_{k+1}\delta_k + \sum_{k=0}^{N-1}\widetilde{\delta}_k.
\end{align*}
The last inequality proves \eqref{Th:fast_str_conv_adap:result_1}.
Inequality \eqref{Th:fast_str_conv_adap:result_2} is a straightforward from \eqref{Th:str_conv_adap:proof:ineq_1} since~$f(x) \geq f(x_*)$ for all~$x \in Q$.
\end{proof}

To finish the analysis of Algorithm \ref{alg:FastAlg2_strong} we estimate the growth rate of the sequence~$A_N$. The result is proved in the same way as Lemma 3.7 in~\cite{stonyakin2021inexact} with the change~$L_{k} \to M_{k}$.
\begin{lemma}
\label{lemma:a_n_sequence}
For all~$N\geq 0$, we have
\begin{align*}
A_N &\geq \max\left\{\frac{1}{4}\left(\sum_{k=0}^{N-1}\frac{1}{\sqrt{M_{k+1}}}\right)^2, \frac{1}{M_1}\prod_{k=1}^{N-1}\left(1 + \sqrt{\frac{\mu + m}{4M_{k+1}}}\right)^2\right\}\\
& \geq \max\left\{\frac{N^2}{4\widetilde{M}_N}, \frac{1}{M_1}\exp\left(N\sqrt{\frac{\mu + m}{4\widetilde{M}_{N}}}\right)\right\},
\end{align*}
where~$\widetilde{M}_N^{-1/2} = \frac{1}{N}\sum_{k=0}^{N-1}M_{k+1}^{-1/2}$.
\end{lemma}
Note that from \eqref{eq:M_bound} we have that~$\widetilde{M}_N^{-1/2} = \frac{1}{N}\sum_{k=0}^{N-1}M_{k+1}^{-1/2} \geq \frac{1}{\sqrt{2LG}}$, which leads to the following estimate for the convergence rate of Algorithm \ref{alg:FastAlg2_strong}
\begin{align*}
f(x_N) - f(x_*) \leq &D_{\phi}[u_0](x_*) \min\left\{\frac{8LG}{N^2}, 2LG\exp\left(-N\sqrt{\frac{\mu + m}{8LG}}\right)\right\} \\
& + \frac{2\sum_{k=0}^{N-1}A_{k+1}\delta_k}{A_N} + \frac{\sum_{k=0}^{N-1}\widetilde{\delta}_k}{A_N}.
\end{align*}

\section{Inexact Statistically Preconditioned Accelerated Gradient Method}\label{sec:problem}

In this section, we return to the distributed empirical risk minimization problem~\eqref{eq:ERM1}, \eqref{eq:ERM2}, where we deal with~$m$ machines or worker nodes, with sample size~$n$ at each. Moreover, without loss of generality we index the central node as node~$1$. Following the same algorithmic structure as DANE~\cite{pmlr-v32-shamir14} and SPAG~\cite{pmlr-v119-hendrikx20a}, we define a reference function 
\begin{equation}
    \label{eq:phi_def}
    \phi(x) = \frac{1}{n}\sum_{i=1}^{n} \ell(x;\mathbf{\zeta}_{i}) + \frac{\sigma}{2}\|x\|_2^2,
\end{equation}
where the samples~$\mathbf{\zeta}_{i}$ are taken from the node which is chosen to be central. 
It is easy to see from \eqref{eq:ERM1_1} and \eqref{eq:rel_str_conv_simple} that ~$\phi(x)$ is~$L_{\phi}$-smooth, and~$\mu_{\phi}$-strongly convex since it has a similar form as~$F(x)$. The value of the parameter~$\sigma$ is set to be an upper bound that quantifies how similar the function~$F_1$ is to~$F$, i.e., we assume that with high probability, it holds that
\begin{equation}\label{eq:stat}
    \|\nabla^2 F(x) - \nabla^2 F_1(x)\|_2\le \sigma, \; \forall x \in {\rm dom } h
\end{equation}
where the norm is the operator norm for
matrices (i.e., the largest singular value).
The rationale behind this statistical similarity assumption are statistical arguments  that allow to show~\cite{pmlr-v119-hendrikx20a} that \eqref{eq:stat} holds with~$\sigma$ proportional to~$\frac{1}{\sqrt{n}}$.
Further, it follows that~$F(x)$ is~$L_{F/\phi}$-relatively smooth and~$\mu_{F/\phi}$-relatively strongly convex with respect to~$\phi(x)$~\cite{pmlr-v37-zhangb15,pmlr-v119-hendrikx20a}, i.e., 
\begin{align}\label{eq:relative_phi}
   \mu_{F/\phi} D_\phi[x](y)\leq D_F[x](y) \leq L_{F/\phi}  D_\phi[x](y),
\end{align}
with~$L_{F/\phi}=1$,~$\mu_{F/\phi}={\mu_F}/({\mu_F+2\sigma})$, and~$\kappa_{F/\phi} = {L_{F/\phi}}/{\mu_{F/\phi}}=1+2\sigma/\mu_F$.

Once the specific Bregman divergence has been defined based on statistical similarity and using the reference function (statistical preconditioner)~$\phi(x)$ as in \eqref{eq:phi_def}, distributed statistical preconditioning methods rely on Bregman proximal steps, where the algorithm needs to solve at every iteration the problem of the form (here~$\alpha >0$)
\begin{align}\label{eq:GDRS}
    {\argmin_{x\in\mathbb{R}^d}} \left\{\alpha(\langle\nabla F(z), x {-} z\rangle + h(x)) {+}  D_\phi[u](x)\right\}. 
\end{align}

Non-accelerated methods based on steps of the form~\eqref{eq:GDRS} have an iteration complexity of~$\widetilde{O}(\kappa_{F/\phi})$~\cite{bauschke2017descent,lu2018relatively,stonyakin2021inexact}.
Thus, statistical preconditioning allows for the relative condition number~$\kappa_{F/\phi}$ to determine the convergence rate instead of~$\kappa_{F}$. The authors in~\cite{pmlr-v119-hendrikx20a} showed that for quadratic functions~$\sigma= \widetilde{O}(L_F/\sqrt{n})$, which implies~$\kappa_{F/\phi}=1+\widetilde{O}(\kappa_F/\sqrt{n})$. Similarly, for non-quadratic functions~$\sigma= \widetilde{O}(\kappa_F\sqrt{d/n})$, thus~$\kappa_{F/\phi}=1+\widetilde{O}(\kappa_F\sqrt{d/n})$.
This, in turn, leads to the total number of communication rounds~$\widetilde{O}\left(\kappa_{F/\phi}\right)$, which is quantitatively better than for methods that do not use such statistical preconditioning~\cite{arjevani,scaman2017optimal,hendrikx2020optimal}. A similar argument follows for  accelerated algorithms, where the iteration complexity will be~$\widetilde{O}\big(\kappa_{F/\phi}^{1/2}\big)$~\cite{pmlr-v119-hendrikx20a}. 

Next, we study the building blocks of our approach to advance this line of works. First, we consider the inexact version of the SPAG algorithm~\cite{pmlr-v119-hendrikx20a} wherein each iteration subproblems of the form \eqref{eq:GDRS} are solved inexactly with such accuracy that the overall performance of the algorithm is affected only by a logarithmic factor. Notably, the required accuracy decreases as iterations go, meaning that the approximate solution's quality may not be high in the first iterations. Next, we introduce and analyze a Hyperfast second-order method for third-order smooth and uniformly convex functions, which we will apply to solve subproblems \eqref{eq:GDRS} in each iteration of our inexact SPAG (InSPAG) algorithm when~$h(x)=0$. Finally, we analyze the total complexity for the combination of InSPAG plus the Hyperfast second-order method to solve our problem of interest. This combination is advantageous because we only use first-order information on the individual losses from the whole dataset and obtain a small subproblem on the central node. Then, a fast second-order method is used to solve this subproblem on the central node.

\subsection{InSPAG and Its Convergence Rate Theorem}\label{sec:method}

\begin{algorithm}[h!]
\caption{InSPAG~$(L_{F / \phi}, \mu_{F/\phi}, x_0, R)$}
\label{algo:inspag}
\begin{algorithmic}[1]
\STATE \textbf{Input:}~$R$ s.t.~$x_* \in B_2(0,R)$,~$R_\phi^2 = 2L_{\phi}R^2$,~$\mu_{F/\phi}$,~$M_{0}$. 
\STATE Set
$y_0 = u_0= x_0 \in B_2(0,R)$,~$A_0:= \alpha_0 := 0$.
\FOR{$k \geq 0$ }
\STATE Set~$i_k=0$
\REPEAT
\STATE \textbf{At the central node} set~$M_{k+1} = 2^{i_k-1}M_k$ and find~$\alpha_{k+1}$ from~$
A_{k+1}{(1 + A_k \mu_{F/\phi})}=M_{k+1}\alpha^2_{k+1}$. Set~$A_{k+1} := A_k + \alpha_{k+1}$.
\STATE \textbf{At the central node} set~$y_{k+1} := \frac{\alpha_{k+1}u_k + A_k x_k}{A_{k+1}}$ and send to each worker.
\STATE \textbf{At every worker node~$j$} compute~$\frac{1}{n}\sum_{i=1}^n {\nabla} \ell\bigl(y_{k+1}; \zeta^{(j)}_i\bigr)$ and send it to the central node.
\STATE \textbf{At the central node} compute
$\nabla F(y_{k+1}) = \frac{1}{nm}\sum_{j=1}^m\sum_{i=1}^n {\nabla} \ell\bigl(y_{k+1}; \zeta^{(j)}_i\bigr)$.
\STATE \label{step:subproblem_SPAG} \textbf{At the central node} solve~$u_{k+1} = \arg \min_{x \in B_2(0,R)}^{R_\phi^2/k}\Phi_{k+1}(x)$, 
\begin{align}
\text{where }\; & \Phi_{k+1}(x)  = \alpha_{k+1} (\langle \nabla F(y_{k+1}), x - y_{k+1}\rangle + h(x)) + \nonumber \\
& \qquad + {(1 + A_k\mu_{F/\phi})} D_\phi[u_k](x)  +\alpha_{k+1} \mu_{F/\phi}  D_\phi[y_{k+1}](x). \label{eq:V_t_min_step}
\end{align}
\vspace{-2em}
\STATE \textbf{At the central node} set~$x_{k+1} := \frac{\alpha_{k+1}u_{k+1} + A_k x_k}{A_{k+1}}$.
\STATE Set~$i_k = i_k+1$.
\UNTIL{
\begin{equation}
\label{eq:inspag_crit}
F(x_{k+1}) \leq F(y_{k+1}) + \langle \nabla F(y_{k+1}), x_{k+1} - y_{k+1} \rangle +\frac{M_{k+1}\alpha^2_{k+1}}{A_{k+1}^2}D_{\phi}[u_{k}](u_{k+1}).
\end{equation}
}
\ENDFOR
\STATE \textbf{Ouput:}~$x_k$
\end{algorithmic}
\end{algorithm}

This subsection introduces the InSPAG algorithm together with its convergence rate analysis. The main idea is to implement Algorithm \ref{alg:FastAlg2_strong} on the central node and use Theorem \ref{Th:fast_str_conv_adap}. Inexactness in statistically preconditioned problems has been studied for DANE, resulting in InexactDANE, AIDE~\cite{reddi2016aide}, and D$^2$ANE~\cite{JMLR:v21:19-764}. 
To propose our InSPAG algorithm we rely on the results of Section \ref{S:AGM}. From \eqref{eq:relative_phi} and Examples \ref{Ex:rel_smooth}  and \ref{Ex:composite} we see that~$f_{\delta}(y) = f(y)$ and~$\psi_\delta(x,y) = \langle \nabla F(y), x - y \rangle + h(x) - h(y)$ constitute a~$(0, L_{F/\phi}, \mu_{F/\phi}, 0, \phi)$-model of the function~$f$ defined in \eqref{eq:ERM1}. Thus, the main idea of InSPAG is to implement Algorithm \ref{alg:FastAlg2_strong} for problem \eqref{eq:ERM1} using distributed computations. We further assume that the solution~$x_{*}$ of the problem \eqref{eq:ERM1} belongs to some Euclidean ball~$B_2(0,R)$, and define~$R_{\phi}^2 = 2 L_{\phi}R^2$. Using this quantity we set the inexactness of the projection in each iteration to be~$\widetilde{\delta}_k = \frac{R_{\phi}^2}{k}$ (cf. \eqref{equmir2DL_strong}).

The pseudocode of the proposed InSPAG algorithm is presented as Algorithm~\ref{algo:inspag}. Unlike~\cite{pmlr-v119-hendrikx20a}, our algorithm is inspired by a similar-triangles type of accelerated methods~\cite{gasnikov2018universal,Nesterov2018,dvurechensky2018computational,dvurechensky2018decentralize,stonyakin2021inexact,dvurechensky2021first-order}, which leads to a slightly simpler algorithm.
Another important difference with~\cite{pmlr-v119-hendrikx20a} is that our algorithm is adaptive simultaneously to the constants~$L_{F/\phi}$ and~$G$ (see Assumption \ref{assum:Bregman}), which may lead to further acceleration in practice since locally, the constant~$L_{F/\phi}G$ can be smaller leading to larger step-sizes.
Note that Line \ref{step:subproblem_SPAG} of Algorithm~\ref{algo:inspag} requires approximate minimization of the auxiliary function~\eqref{eq:V_t_min_step}. First, we present the complexity analysis of Algorithm~\ref{algo:inspag} in Theorem~\ref{Th:InSPAG_conv} assuming the approximate solution to~\eqref{eq:V_t_min_step}. In Subsection~\ref{subsec:inner}, we show the complexity of obtaining said approximate solution efficiently when~$h(x)=0$ using high-order methods. 

We are now in a position to state the main result on InSPAG.
\begin{theorem}\label{Th:InSPAG_conv}
Assume that the function~$F$ in \eqref{eq:ERM1} is~$\mu_{F/\phi}$-strongly convex and~$L_{F/\phi}$-smooth with respect to the function~$\phi$, where~$\phi$ satisfies Assumption \ref{assum:Bregman}.  Moreover, let~$x_k$,~$k \geq 0$ be the sequence generated by Algorithm~\ref{algo:inspag}. Then, after~$K$ iterations it holds that
\begin{align}
&f(x_K) - f(x_*) \leq 
\frac{2L_{\phi}R^2(1+\ln K)}{A_K}, \label{eq:InSPAG_rate1}
\end{align}
Moreover, the value~$A_K$ grows as follows:
\begin{align}
\label{eq:A_k_rate_SPAG}
A_K &\geq \max\left\{\frac{K^2}{4\widetilde{M}_K}, \frac{1}{M_1}\exp\left(K\sqrt{\frac{\mu_{F/\phi}}{4\widetilde{M}_{K}}}\right)\right\},
\end{align}
where~$\widetilde{M}_K^{-1/2} = \frac{1}{K}\sum_{k=0}^{K-1}M_{k+1}^{-1/2}$.
\end{theorem}
\begin{proof}
Clearly, Algorithm \ref{algo:inspag} is a distributed implementation of Algorithm \ref{alg:FastAlg2_strong} with~$\delta_k = 0$,~$k\geq 0$. We only note that for this particular setting with~$f_{\delta}(y) = f(y)$  and~$\psi_\delta(x,y) = \langle \nabla F(y), x - y \rangle + h(x) - h(y)$, inequality  \eqref{exitLDL_strong} becomes
\begin{align*}
F(x_{k+1})+h(x_{k+1}) \leq & F(y_{k+1})+h(y_{k+1}) + \langle \nabla F(y_{k+1}), x_{k+1} - y_{k+1} \rangle \\
&+ h(x_{k+1}) - h(y_{k+1}) +\frac{M_{k+1}\alpha^2_{k+1}}{A_{k+1}^2}D_{\phi}[u_{k}](u_{k+1}),
\end{align*}
which is equivalent to \eqref{eq:inspag_crit}. 
Thus, we can apply Theorem \ref{Th:fast_str_conv_adap}, which gives the following estimate
\begin{align*}
f(x_K) - f(x_*) &\leq \frac{D_{\phi}[u_0](x_*)}{A_K} +  \frac{\sum_{k=0}^{K-1}\widetilde{\delta}_k}{A_K} 
\leq \frac{L_{\phi}(2R)^2}{2A_K} + \frac{1}{A_K}\sum_{k=0}^{K-1}\frac{R_\phi^2}{k} \\
& \leq \frac{R_\phi^2(1+\ln K)}{A_K} = \frac{2L_{\phi}R^2(1+\ln K)}{A_K} 
\end{align*}
The lower bound for~$A_{K}$ follows from Lemma \ref{lemma:a_n_sequence}.
\end{proof}

To apply Theorem \ref{Th:InSPAG_conv} we need to ensure that Assumption \ref{assum:Bregman} is satisfied.  
\begin{lemma}
\label{Lm:check_Bregman_assumpt}
Under the assumption that~$\phi$ is~$\mu_\phi$-strongly convex and~$L_\phi$-smooth Assumption \ref{assum:Bregman} is satisfied with~$G=L_\phi/\mu_\phi = \kappa_{\phi}$.
\end{lemma}
\begin{proof}
Since~$\phi$ is~$\mu_\phi$-strongly convex and~$L_\phi$-smooth, we have that
\[
\frac{\mu_\phi}{2}\|x-y\|^2 \leq D_{\phi}[x](y) \leq \frac{L_\phi}{2}\|x-y\|^2, \;\;\forall x,y \in \text{{\rm dom }} \phi.
\]
Thus, for all~$x,y,u,u_{+}$ such that~$x-y=\tau(u_+-u)$ for some~$\tau \in [0,1]$, we have
\begin{align*}
D_{\phi}[y](x) &\leq \frac{L_\phi}{2}\|x -y \|^2 = 
\frac{L_\phi \tau^2}{2 } \|u_{+}-u\|^2 \leq
\frac{L_\phi\tau^2}{\mu_\phi}D_{\phi}[u](u_+).
\end{align*}
\end{proof}

From Lemma~\ref{Lm:check_Bregman_assumpt}, we see that if~$\phi$ is a quadratic function, then,~$G=\kappa_{\phi}$ and by \eqref{eq:M_bound} we have that~$M_{k+1}\leq 2 L_{F/\phi}\kappa_{\phi}$. Then, the number of iterations~$K$ to reach accuracy~$\e$, i.e., the number of communications between the central node and the worker nodes, is bounded as~$O(\sqrt{\kappa_{F/\phi}\kappa_{\phi}}\ln\frac{1}{\e})$. As we see below, for quadratic functions the estimate for~$G$ can be improved to~$G=1$, which gives a better communication complexity~$O(\sqrt{\kappa_{F/\phi}}\ln\frac{1}{\e})$. In the general case, where~$\phi$ is not quadratic, similarly to~\cite{pmlr-v119-hendrikx20a, pmlr-v32-lin14}, we next show that~$M_{k+1} \to L_{F/\phi}$ linearly with rate~$\widetilde{O}(\sqrt{\kappa_{F/\phi}})$. This means that the convergence rate of InSPAG quickly approaches the convergence rate with condition number~$\sqrt{\kappa_{F/\phi}}$.

\begin{lemma}
\label{Lm:M_k_when_Lip_Hess}
Under the assumptions of Theorem \ref{Th:InSPAG_conv} and Lemma \ref{Lm:check_Bregman_assumpt} assume additionally that the Hessian of~$\phi$ is~$H$-Lipschitz-continuous, i.e.
\begin{equation}
    \label{eq:phi_Lip_Hess}
    \|\nabla^2 \phi(x) - \nabla^2 \phi(y) \| \leq H \|x-y\|.
\end{equation}
Then the inequality \eqref{eq:inspag_crit} is satisfied with
\begin{equation}
    \label{eq:M_k_when_Lip_Hess}
    M_{k+1} = L_{F/\phi} \min\left\{\kappa_{\phi},1 + \frac{H d_k}{\mu_{\phi}}\right\} ,
\end{equation}
where~$d_k = \|x_{k+1}- y_{k+1}\| + \|u_k - x_k\| + \|u_k-u_{k+1}\|~$.
\end{lemma}
\begin{proof}
By the local quadratic representation of the Bregman divergence, we have for any~$a,b \in {\rm dom } \, \phi$ and for some~$\tau \in [0,1]$ that
$D_{\phi}[a](b) = \|a-b \|_{\nabla^2\phi(\tau a + (1-\tau)b)}^2$. We use~$H(a,b)$ to denote the corresponding Hessian~$\nabla^2\phi(\tau a + (1-\tau)b)$.
We have 
\begin{align*}
&D_{\phi}[x_{k+1}](y_{k+1}) = \|x_{k+1}-y_{k+1} \|_{H(x_{k+1},y_{k+1})}^2 \stackrel{\eqref{eqxmir2DL_strong},\eqref{eqymir2DL_strong}}{=} \frac{\alpha_{k+1}^2}{A_{k+1}^2} \|u_{k+1}-u_{k} \|_{H(x_{k+1},y_{k+1})}^2  \\
& \leq \frac{\alpha_{k+1}^2}{A_{k+1}^2} \left(\|u_{k+1}-u_{k} \|_{H(u_{k+1},u_{k})}^2 + \|H(x_{k+1},y_{k+1}) - H(u_{k+1},u_{k}) \|\| u_{k+1} -u_{k}\|^2 \right) \\
& \leq \frac{\alpha_{k+1}^2}{A_{k+1}^2} \left(D_{\phi}[u_{k}](u_{k+1}) + \|H(x_{k+1},y_{k+1}) - H(u_{k+1},u_{k}) \| \frac{D_{\phi}[u_{k}](u_{k+1})}{\mu_{\phi}}\right) \\
&\stackrel{\eqref{eq:phi_Lip_Hess}}{\leq} \frac{\alpha_{k+1}^2}{A_{k+1}^2} D_{\phi}[u_{k}](u_{k+1})\left(1 + \frac{H\|z-z'\|}{\mu_{\phi}}\right),
\end{align*}
where~$z \in [x_{k+1},y_{k+1}]$ and~$z'\in [u_{k+1},u_{k}]$. Using the latter and \eqref{eqymir2DL_strong}, \eqref{eqxmir2DL_strong}, we obtain
\begin{align*}
&\|z-z'\| \leq \|z- y_{k+1}\| + \|y_{k+1} - u_k\| + \|u_k-z'\|  \\
& \leq \|x_{k+1}- y_{k+1}\| + \|x_k - u_k\| + \|u_k-u_{k+1}\| \triangleq d_k.
\end{align*}
Combining the above with the relative smoothness property \eqref{eq:relative_phi}, we obtain that \eqref{eq:inspag_crit} holds when~$M_{k+1}=L_{F/\phi}\left(1 + \frac{H d_k}{\mu_{\phi}}\right)$. Since \eqref{eq:inspag_crit} holds also when~$M_{k+1}=L_{F/\phi}\kappa_{\phi}$ (see Lemma \ref{Lm:check_Bregman_assumpt} and \eqref{eq:relative_phi}), we obtain the statement of the Lemma.
\end{proof}

From \eqref{Th:fast_str_conv_adap:result_2} and \eqref{eq:A_k_rate_SPAG} since~$M_{k+1}\leq L_{F/\phi}\kappa_{\phi}$ we know that the sequence~$u_k$,~$k\geq 0$ converges to~$x_*$ linearly with condition number~$\sqrt{\kappa_{F/\phi}\kappa_{\phi}}$. From \eqref{Th:fast_str_conv_adap:result_1} by the strong convexity, we see that the sequence~$x_k$,~$k\geq 0$ converges to~$x_*$ also linearly with the same condition number. Hence, by \eqref{eqymir2DL_strong} we conclude the same on the sequence~$y_k$,~$k\geq 0$. Thus,~$d_k$ converges linearly to zero with the same condition number and~$M_{k+1}$ approaches~$L_{F/\phi}$ with the same rate.
This, in turn, means that the convergence rate in Theorem \ref{Th:InSPAG_conv} quickly approaches~$O((1-\sqrt{\kappa_{F/\phi}})^K)$ when the Hessian of~$\phi$ is Lipschitz-continuous. 

Next, we study the properties of the auxiliary problem in step \ref{step:subproblem_SPAG} of Algorithm \ref{algo:inspag} and, under the additional assumption that the loss function~$\ell$ has bounded fourth-order derivatives, we show the explicit complexity of computing an approximate solution to this auxiliary problem using Hyperfast second-order methods. 

\subsection{Hyperfast Second-Order Method for the Auxiliary Problem}\label{subsec:inner}

In this subsection, we elaborate the properties of the auxiliary problem in step \ref{step:subproblem_SPAG} of Algorithm \ref{algo:inspag} and propose a Hyperfast second-order algorithm to solve it when the function~$\phi$ is strongly convex and sufficiently smooth. The main result is a complexity estimate for solving the auxiliary problem by the Hyperfast algorithm. Recall that, at each iteration of Algorithm~\ref{algo:inspag} we need to find an approximate minimizer in the sense of Definition~\ref{Def:solNemirovskiy} of the function~$\Phi_{k+1}(x)$ on the Euclidean ball~$B_2(0,R)$. 
Throughout this subsection we assume that the regularizer~$h(x) \equiv 0$. 

We first study some properties of the function~$\Phi_{k+1}(x)$ defined in \eqref{eq:V_t_min_step} and the minimization problem solved in 
step \ref{step:subproblem_SPAG} of Algorithm \ref{algo:inspag}. Using our assumption that~$h(x)=0$, the fact that~$A_{k+1}=A_k+\alpha_{k+1}$, the definition of the Bregman divergence, and ignoring constant terms in that problem, we see that it is equivalent to the problem 
$u_{k+1} = \arg\min_{x \in B_2(0,R)}^{R_\phi^2/k} \Psi_{k+1}(x)$,  where
\begin{align}
& \Psi_{k+1}(x)  \triangleq  \langle \alpha_{k+1}\nabla F(y_{k+1}) - (1 + A_k\mu_{F/\phi}) \nabla \phi (u_k) -  \alpha_{k+1} \mu_{F/\phi} \nabla \phi (y_{k+1}), x\rangle  + \nonumber \\
& \qquad + (1 + A_{k+1}\mu_{F/\phi})\phi(x). \label{eq:Psi_min}
\end{align}
\begin{lemma}
\label{Lm:inexact_func_to_Nemir}
Assume that~$\phi$ is~$\mu_\phi$-strongly convex and~$L_\phi$-smooth w.r.t. the Euclidean norm. Also assume that for some~$\theta>0$ and all $x \in B_2(0,R)$, with~$\max\{ \|\nabla F(x)\|_2/\mu_{F/\phi}, \|\nabla \phi(x)\|_2\} \leq \theta$.
Let us denote \\$x_{k+1}^* = \arg \min_{x \in B_2(0,R)} \Psi_{k+1}(x)$ and let the point~$\hat{x}_{k+1}$ satisfy 
\begin{align}
    \label{eq:V_tolerance}
   \hspace{-0.15cm} \Psi_{k+1}(\hat{x}_{k+1}) {-} \Psi_{k+1}(x_{k+1}^*) \leq \Delta_k \triangleq \frac{\mu_\phi R_\phi^4}{2k^2(2L_{\phi}R +3\theta )^2( 1 + A_{k+1}\mu_{F/\phi})}.
\end{align}
Then~$\hat{x}_{k+1} = \arg\min_{x \in B_2(0,R)}^{R_\phi^2/k} \Psi_{k+1}(x)$.
\end{lemma}
\begin{proof}
Since~$\phi$ is~$\mu_\phi$-strongly convex and~$L_\phi$-smooth,~$\Psi_{k+1}$ in \eqref{eq:Psi_min} is~$\mu_{\Psi}$-strongly convex with~$\mu_{\Psi}=( 1 + A_{k+1}\mu_{F/\phi} )\mu_\phi$ and~$L_\Psi$-smooth with~$L_{\Psi}=( 1 + A_{k+1}\mu_{F/\phi} )L_\phi$. Further, by the assumption of the lemma, we have, for all~$x \in B_2(0,R)$,
\begin{align}
    &\| \nabla \Psi_{k+1} (x)  \|_2 = \| \alpha_{k+1}\nabla F(y_{k+1}) - (1 + A_k\mu_{F/\phi}) \nabla \phi (u_k) -  \alpha_{k+1} \mu_{F/\phi} \nabla \phi (y_{k+1}) \nonumber\\
    & + (1 + A_{k+1}\mu_{F/\phi}) \nabla \phi(x) \|_2 
    \leq 3(1 + A_{k+1}\mu_{F/\phi}) \theta, \label{eq:Lm:inexact_func_to_Nemir_proof_1}
\end{align}
where we used also that~$\alpha_{k+1}\leq A_{k+1}$ and that~$A_{k+1}=A_k+\alpha_{k+1}$.
By the strong convexity of~$\Psi$, we have
\begin{align}
\|\hat{x}_{k+1} - x_{k+1}^*\|_2\leq   \sqrt{\frac{2}{\mu_{\Psi}}(\Psi_{k+1}(\hat{x}_{k+1}) -\Psi_{k+1}(x_{k+1}^*))} \leq \sqrt{2\Delta_k/\mu_\Psi}.
    \label{eq:Lm:inexact_func_to_Nemir_proof_2}
\end{align}
Hence, for any~$x \in B_2(0,R)$,
\begin{align*}
    &\la \nabla \Psi_{k+1}(\hat{x}_{k+1}), x - \hat{x}_{k+1}\ra = \la \nabla \Psi_{k+1}(\hat{x}_{k+1}) - \nabla \Psi_{k+1}(x_{k+1}^*), x - \hat{x}_{k+1}\ra \\
    &+\la \nabla \Psi_{k+1}(x_{k+1}^*), x - x_{k+1}^* \ra + \la \nabla \Psi_{k+1}(x_{k+1}^*), x_{k+1}^* - \hat{x}_{k+1} \ra \notag \\
    &\geq - L_{\Psi} \|x_{k+1}^*- \hat{x}_{k+1}\|_2\|x - \hat{x}_{k+1}\|_2 + 0 - \|\nabla \Psi_{k+1}(x_{k+1}^*)\|_2 \|x_{k+1}^* - \hat{x}_{k+1}\|_2 \\
    &\stackrel{\eqref{eq:Lm:inexact_func_to_Nemir_proof_1},\eqref{eq:Lm:inexact_func_to_Nemir_proof_2}}{\geq} - (2L_{\Psi}R+3(1 + A_{k+1}\mu_{F/\phi}) \theta) \sqrt{2\Delta_k/\mu_\Psi} \\
    & = -(1 + A_{k+1}\mu_{F/\phi})(2L_{\phi}R+3\theta)  \sqrt{\frac{2\Delta_k}{( 1 + A_{k+1}\mu_{F/\phi} )\mu_\phi}} \geq - R_\phi^2/k
\end{align*}
where we used the definitions of~$L_{\Psi}$ and~$\mu_{\Psi}$ and the expression for~$\Delta_k$. Thus,~$\hat{x}_{k+1}$ satisfies Definition \ref{Def:solNemirovskiy} with~$\widetilde{\delta} = R_\phi^2/k$.
\end{proof}

Next, we propose an efficient Hyperfast second-order method to obtain a point~$\hat{x}_{k+1}$ for which~\eqref{eq:V_tolerance} holds. To do this, we make an additional assumption on the function~$\phi$.
\begin{assumption}\label{assum:high}
The function~$\phi$ has bounded fourth-order derivatives, which is equivalent to Lipschitz third-order derivative, i.e. there exists~$0\leq L_{\phi,3} < \infty$ s.t.
\begin{align*}
    \|\nabla^3\phi(x)-\nabla^3\phi(y) \|_2 \leq  L_{\phi,3} \|x-y\|_2, \;\; \forall x,y \in B_2(0,R),
\end{align*}
where the  norm of a tensor is induced by the Euclidean norm in a standard way~\cite{nesterov2019implementable}.
\end{assumption}

The idea is to use a second-order implementation of a third-order method, in the sense of~Sect. 5.2 from \cite{nesterov2020inexact} or~ Algorithm 2 \cite{kamzolov2020near}, to minimize~$\Psi_{k+1}(x)$ in each iteration of InSPAG. Such methods are called Hyperfast second-order methods since, due to the additional assumption of third-order smoothness, they have faster convergence rates than the optimal second-order method~\cite{monteiro2013accelerated}.
In our case, the objective~$\Psi_{k+1}(x)$ is additionally strongly convex. Thus, we can achieve faster rates than the basic schemes in~\cite{nesterov2020inexact,kamzolov2020near} that do not use strong convexity. We propose an extension of
Hyperfast second-order methods for minimizing strongly convex functions and show that they have faster convergence rate.\footnote{Section~\ref{app:uni} extends  Hyperfast second-order methods for a more general setting of minimizing uniformly convex functions. Here we use a particular case that corresponds to uniform convexity of the order~$q=2$, equivalent to strong convexity.} 
Our algorithm is described below as Algorithm \ref{alg:rest-Hyp-dist-V}. 

\begin{algorithm}
\caption{Restarted Hyperfast Second-Order Method}
\label{alg:rest-Hyp-dist-V}
\begin{algorithmic}[1]
   \REQUIRE~$z_0 \in B_2(0,R)$, constant~$c$ which defines convergence rate of the basic Hyperfast method, strong convexity parameter~$\mu_{\phi}$.
   \STATE Set~$R_0=2R$
   \FOR{$t=0,1,...$}
			\STATE Set ~$R_t = R_0 \cdot 2^{-k}$, and 
			$
			     N_t = \max \{ \lceil \big({8 c L_{\phi,3} R_t^{2}}/{\mu_\phi} \big)^\frac{1}{5} \rceil, 1 \},
			$
			
			\STATE Set~$z_{t+1}=y_{N_t}$ as the output of the basic Hyperfast Second-Order Method (either~Eq.3.6 \cite{nesterov2020inexact} for~$p=3$ and~$\beta=1/2$ and with auxiliary steps described in~Sect. 5.2 from \cite{nesterov2020inexact} or~Algorithm 2 from \cite{kamzolov2020near}) started from~$z_t$ and run for~$N_t$ steps applied to~$\Psi_{k+1}(x)$.
			\STATE Set~$t=t+1$.
			\ENDFOR	
		\ENSURE~$z_t$.
\end{algorithmic}
\end{algorithm}

As a building block, this method uses basic Hyperfast second-order method which has  convergence rate of the form~${cL_3\|x_*-z_0\|_2^4}/{k^5}$, where~$k$ is the iteration counter,~$c = 48$ for~Theorem 2 from \cite{nesterov2020inexact} and~$c = 35$ for~Theorem 2 from \cite{kamzolov2020near}.

\begin{theorem}
\label{Th:Hyper_V_compl}
Under assumptions of Lemma \ref{Lm:inexact_func_to_Nemir} let additionally Assumption \ref{assum:high} to hold. Let also sequence~$z_t$,~$t \geq 0$ be generated by Algorithm \ref{alg:rest-Hyp-dist-V}. Then 
\begin{equation}
\label{eq:Hyper_V_rate}
    \frac{\mu_\Psi}{2} \|z_{t} - x_{k+1}^*\|_2^2 \le \Psi_{k+1}(z_{t}) - \Psi_{k+1}(x_{k+1}^*) \le  2\mu_\Psi R^2 \cdot 2^{-2t}, t\geq 0.
\end{equation}
Moreover, the total number of steps of the basic Hyperfast second-order method to reach~$\Psi_{k+1}(z_{t}) - \Psi_{k+1}(x_{k+1}^*) \leq \Delta_k$ is bounded by 
\[
5\bigg( \frac{32 c L_{\phi,3} R^{2}}{\mu_{\phi}} \bigg)^\frac{1}{5} +\log_2\frac{( 1 + A_{k+1}\mu_{F/\phi})^2k^2(2L_{\phi}R +3\theta )^2 }{L_{\phi}^2 R^2}.
\]
\end{theorem}
\begin{proof}
Let us denote for shortness~$x^*=x_{k+1}^*$ and~$\Psi(x) = \Psi_{k+1}(x)$. 
For~$t = 0$ we have ~$\|x^* - z_0\|_2 \le R_0$. 
Let us assume that~$\|x^* - z_t\|_2 \le R_t$ and show that~$\|x^* - z_{t + 1}\|_2 \le R_{t + 1}$. 
By Assumption \ref{assum:high} and \eqref{eq:Psi_min} it is clear that~$\Psi(x)$ has~$L_{\Psi,3}$-Lipschitz third-order derivative with 
$L_{\Psi,3}=( 1 + A_{k+1}\mu_{F/\phi} )L_{\phi,3}$. Recall that~$\mu_{\Psi}=( 1 + A_{k+1}\mu_{F/\phi} )\mu_\phi$.
From {\cite{nesterov2020inexact}[Theorem 2]} since~$\Psi$ is~$\mu_{\Psi}$-strongly convex and has~$L_{\Psi,3}$-Lipschitz third-order derivative, it holds that
	\[
	\frac{\mu_{\Psi}}{2} \|z_{t + 1} - x^*\|_2^2 \le \Psi(z_{t + 1}) - \Psi(x^*) \le \frac{c L_{\Psi,3} \|z_t - x^*\|_2^{4}}{N_t^{5}} \le \frac{\mu_{\Psi} (R_t / 2)^2}{2} = \frac{\mu_{\Psi} R_{t + 1}^2}{2}
	\]
	by the choice of~$N_t$ and since~$L_{\Psi,3}/\mu_{\Psi}=L_{\phi,3}/\mu_{\phi}$.
	Thus, by induction, we have \eqref{eq:Hyper_V_rate}.
	
	It remains to estimate the number of iterations of the basic Hyperfast method. From \eqref{eq:Hyper_V_rate} we see that to reach the accuracy~$\Delta_k$ it is sufficient to make~$T=\frac{1}{2}\log_2\frac{2\mu_{\Psi}R^2}{\Delta_k}$ restarts. Summing up the number of operations~$N_t,\ t=0,...,T$, we obtain
	\begin{gather*}
	\sum_{t = 0}^T N_t \le \sum_{t = 0}^T \Bigg[ \bigg( \frac{8 c L_{\phi,3} R_t^{2}}{\mu_{\phi}} \bigg)^\frac{1}{5} + 1 \Bigg] = \bigg( \frac{8 c L_{\phi,3} R_0^{2}}{\mu_{\phi}} \bigg)^\frac{1}{5} \sum_{t = 0}^T 2^{-\frac{2t }{5}} + T  \\
	\leq 5\bigg( \frac{32 c L_{\phi,3} R^{2}}{\mu_{\phi}} \bigg)^\frac{1}{5}  + \log_2\frac{2\mu_{\Psi} R^2}{\Delta_k}.
	\end{gather*}
	Let us estimate the last term using \eqref{eq:V_tolerance} and that~$\mu_{\Psi}=( 1 + A_{k+1}\mu_{F/\phi} )\mu_\phi$,~$R_\phi^2 = 2L_{\phi}R^2$:
	\begin{gather*}
	\log_2\frac{2\mu_{\Psi} R^2}{\Delta_k} = \log_2\frac{2( 1 + A_{k+1}\mu_{F/\phi} )\mu_\phi R^2}{\frac{\mu_\phi (2L_{\phi}R^2)^2}{2k^2(2L_{\phi}R +3\theta )^2( 1 + A_{k+1}\mu_{F/\phi})}} \\
	= \log_2\frac{( 1 + A_{k+1}\mu_{F/\phi})^2k^2(2L_{\phi}R +3\theta )^2 }{L_{\phi}^2 R^2}.
	\end{gather*}
	Combining this with the previous chain of inequalities, we obtain the second statement of the lemma.
\end{proof}

\subsection{InSPAG plus Hyperfast Method with Application to Logistic Regression}

This subsection combines the building blocks introduced in the previous two subsections and considers a particular application to a regularized logistic regression problem, for which we obtain a total complexity bound in terms of the number of iterations of the Hyperfast second-order method. We further discuss the arithmetic iteration complexity of our method and compare it to that of stochastic variance-reduced first-order algorithms and indicate a regime in which our algorithm is preferable.

Combining Theorems \ref{Th:InSPAG_conv} and \ref{Th:Hyper_V_compl}, we obtain the following result.
\begin{theorem}
\label{Th:InSPAG+Hyperfast_compl}
Assume that in problem \eqref{eq:ERM1},~$h(x)=0$, and that its solution~$x_*$ belongs to the ball~$B_2(0,R)$. Assume that the function~$F$ in this problem is~$\mu_{F/\phi}$-strongly convex and~$L_{F/\phi}$-smooth with respect to the function~$\phi$, where~$\phi$ satisfies Assumption \ref{assum:Bregman}, is~$\mu_{\phi}$-strongly convex,~$L_\phi$-smooth and has~$L_{\phi,3}$-Lipschitz third-order derivative.  Also assume that for some~$\theta>0$ and all~$x \in B_2(0,R)$, with~$\max\{\|\nabla F(x)\|_2/\mu_{F/\phi}, \|\nabla \phi(x)\|_2\} \leq \theta~$.
Let~$\e >0$ be the target accuracy. Finally, let InSPAG (Algorithm \ref{algo:inspag}) be applied to problem \eqref{eq:ERM1}, and in step \ref{step:subproblem_SPAG} of this algorithm let Restarted Hyperfast method (Algorithm  \ref{alg:rest-Hyp-dist-V}) be applied to solve the auxiliary problem. Then a sufficient number of iterations of the basic  Hyperfast method to find an~$\e$-solution to~\eqref{eq:ERM1} is bounded as
\begin{equation}
\label{eq:InSPAG+Hyperfast_compl}
    O\left(  K \bigg( \frac{L_{\phi,3} R^{2}}{\mu_{\phi}} \bigg)^\frac{1}{5} + K\log_2\frac{\mu_{F/\phi}L_{\phi}R^2(L_{\phi}R +\theta )K\ln K }{L_{\phi} R \e}\right),
\end{equation}
where~$K$ is such that~$\frac{2L_{\phi}R^2(1+\ln (K+1))}{A_{K+1}} \leq \e < \frac{2L_{\phi}R^2(1+\ln K)}{A_{K}}$.
\end{theorem}
\begin{proof}
From \eqref{eq:InSPAG_rate1} we see that InSPAG can be stopped at iteration~$K$ when we have~$\frac{2L_{\phi}R^2(1+\ln (K+1))}{A_{K+1}} \leq \e < \frac{2L_{\phi}R^2(1+\ln K)}{A_{K}}$. Then,~$f(x_{K+1}) - f(x_*) \leq \e$. Also, applying Theorem \ref{Th:Hyper_V_compl} we obtain that the total number of iterations of the basic  Hyperfast method, up to numerical constant multipliers, is bounded by
\begin{align*}
&\sum_{k=0}^K \left(\bigg( \frac{L_{\phi,3} R^{2}}{\mu_{\phi}} \bigg)^\frac{1}{5} +\log_2\frac{( 1 + A_{k}\mu_{F/\phi})k(L_{\phi}R +\theta ) }{L_{\phi} R}\right) \\
& \leq_c K \left(\bigg( \frac{L_{\phi,3} R^{2}}{\mu_{\phi}} \bigg)^\frac{1}{5} + \log_2\frac{( 1 + A_{K}\mu_{F/\phi})K(L_{\phi}R +\theta ) }{L_{\phi} R} \right) = \eqref{eq:InSPAG+Hyperfast_compl},
\end{align*}
where in equality~$\leq_c$ means a usual inequality up to a numerical constant factor.
\end{proof}

From \eqref{eq:A_k_rate_SPAG} and Lemma \ref{Lm:M_k_when_Lip_Hess} we know that when~$\phi$ has also Lipschitz Hessian, it is sufficient to take~$K=O\left(\sqrt{\kappa_{F/\phi}\kappa_{\phi}} \ln \frac{1}{\e}\right)$. Lemma \ref{Lm:M_k_when_Lip_Hess} also implies that for quadratic function~$\phi$ it is sufficient to take~$K=O\left(\sqrt{\kappa_{F/\phi}} \ln \frac{1}{\e}\right)$ and that for non-quadratic function~$\phi$ the result is the same up to a fast asymptotic. In the language of the individual loss~$\ell$ and the number of samples~$n$ used for preconditioning, our result is the same~$\widetilde O (\sqrt{\kappa_\ell}/n^{1/4})$ as for the exact algorithm~\cite{pmlr-v119-hendrikx20a}.  Thus, the total number of iterations of the basic  Hyperfast method to find an~$\e$-solution to \eqref{eq:ERM1} can be bounded as
\begin{equation}
\label{eq:SPAG+Hyper_bound_general}
    \widetilde{O}\left(\sqrt{\kappa_{F/\phi}}
     \bigg( \frac{L_{\phi,3} R^{2}}{\mu_{\phi}} \bigg)^\frac{1}{5}   \right).
\end{equation}

So far, we have not explicitly used the finite-sum structure of problem \eqref{eq:ERM1}, \eqref{eq:ERM1_1} and the statistical similarity~\eqref{eq:stat}. In order to do this, we consider the sparse empirical risk minimization problem with regularized logistic loss, where in \eqref{eq:ERM1_1}, for~$i \in \{1,\ldots,N\}$,  
\begin{align}\label{eq:log_reg}
\ell(x; \zeta_{i}) & =  \log \left(1{+}\exp({-}\eta_i \la x ,\mathbf{\xi}_{i}\ra)\right)   {+} \lambda_1 \sum_{j \in I_S}x_j^2{+} \lambda_2 \sum_{j \in I_D}x_j^2,
\end{align}
where~$\mathbf{\zeta}_{i} = (\xi_i,\eta_i)$,~$\eta_i=1$ indicates a positive (clicked) example, and~$\eta_i=-1$ otherwise. We assume there are two types of features, namely, sparse and dense features. Let~$\mathbf{\xi}_{i,j}$ be the~$j$-th element of the vector~$\mathbf{\xi}_i$. Then,~$\mathbf{\xi}_{i,j}$ is a sparse feature if~$\mathbf{\xi}_{i,j} =0$ for almost all~$i \in \{1,\ldots,N\}$, and a dense feature if~$\xi_{i,j} \neq 0$ for many~$i \in \{1,\ldots,N\}$. We denote by~$I_S$ (and~$I_D$) the set of sparse (and dense) features with~$I_S \cup I_D = \{1,\ldots,d\}$ and~$I_S \cap I_D = \emptyset$. Moreover, it follows from~Section 4.4 from \cite{nesterov2005smooth} that in this case the function~$F$ is~$L_F$-smooth with~$L_F=\max\{\lambda_1,\lambda_2\}+\frac{1}{N}\sum_{i=1}^N\|\eta_i\mathbf{\xi}_i\|_2^2=O(s)$, where~$s$ is the average number of nonzero elements in~$\xi_i$, and~$\mu_F$-strongly convex with~$\mu_F= \min\{\lambda_1,\lambda_2\}$. 
For the same reasons, function~$\phi$ defined in \eqref{eq:phi_def} is~$L_{\phi}$-smooth with~$L_{\phi}=\max\{\lambda_1,\lambda_2\}+\frac{1}{n}\sum_{i=1}^n\|\eta_i\mathbf{\xi}_i\|_2^2 + \sigma$ and~$\mu_{\phi}$-strongly convex with ~$\mu_{\phi}= \min\{\lambda_1,\lambda_2\} + \sigma$.
It also has bounded first-, second, and third-order derivatives~\cite{pmlr-v125-bullins20a}. 
More importantly, the logistic loss in~\eqref{eq:log_reg} has bounded fourth-order derivatives~\cite{pmlr-v125-bullins20a}, which means that Assumption \ref{assum:high} holds. 
Indeed, let us define matrix~$A = [\eta_1\mathbf{\xi}_1,\dots,\eta_n\mathbf{\xi}_n]^\top$. Then, by Theorem 5.4 in~\cite{pmlr-v125-bullins20a} with~$\mu=1$ the function~$\frac{1}{n}\sum_{i=1}^{n} \ell(x;\zeta_{i})~$ has Lipschitz third-order derivative with constant~$L_{\ell,3} = 15\|A^\top A\|_2^2$ w.r.t.~$2$-norm or with constant~$L_{\ell,3} = 15$ w.r.t.~$\|\cdot\|_{A^\top A}$-norm. Since adding a quadratic function does not change the  Lipschitz constant for the third-order derivative,~$\phi$  has Lipschitz third-order derivative with constant~$L_{\phi,3}=L_{\ell,3}$.

Applying~\cite[Theorem 3]{pmlr-v119-hendrikx20a}, we obtain that in our setting the statistical similarity parameter in \eqref{eq:stat} is~$\sigma = 1+\widetilde{O}\left( \frac{\max_{i=1,\ldots,n}\|\eta_i\mathbf{\xi}_i\|_2^{3/2}R}{\min\{\lambda_1,\lambda_2\}\sqrt{n}}\right)$ and a sufficient number of InSPAG iterations is~$\widetilde{O}(\sqrt{\kappa_\ell}/n^{1/4})$, which is similar to SPAG~\cite{pmlr-v119-hendrikx20a}. Further, the number of the basic Hyperfast iterations is the same up to a factor
\[
 \bigg( \frac{L_{\phi,3}R^{2}}{\mu_{\phi}}\bigg)^\frac{1}{5}  \leq_c \bigg(\frac{\|A^\top A\|_2^2R^{2}}{\min\{\lambda_1,\lambda_2\}+\sigma}\bigg)^\frac{1}{5} \leq \bigg(\frac{\|A^\top A\|_2^2R^{2}}{\min\{\lambda_1,\lambda_2\}}\bigg)^\frac{1}{5}.
\]

Informally speaking, applying statistical preconditioning allows reducing the minimization of a large sum~$F$ of~$N$ functions in \eqref{eq:ERM1_1} to the  minimization of a moderate sum~$\phi$ of~$n$ functions when  making the step \ref{step:subproblem_SPAG} of Algorithm \ref{algo:inspag}.
To conclude this subsection we would like to discuss the complexity of minimizing function~$\Psi$ in \eqref{eq:Psi_min} which is equivalent to step \ref{step:subproblem_SPAG} of Algorithm \ref{algo:inspag}. To that end, we consider the setting of sparse logistic regression with loss \eqref{eq:log_reg}.
Since~$\phi$ and~$\Psi$ have finite-sum form, a straightforward approach is to apply accelerated variance reduced methods. This leads to arithmetic operations complexity 
\begin{equation}\label{VR}
    \widetilde{O}\left(s\cdot\left(n + \sqrt{n{\kappa}}\right)\right),
\end{equation}
where~$s$ comes from the cost of evaluating a sparse stochastic gradient~$\nabla \ell (x;\zeta_i)$ for some random~$i$, and the rest is the optimal bound on the number of stochastic gradient evaluations for such methods~\cite{lanfirst}. Note that we have~$\kappa=L_{\Psi}/\mu_{\Psi}=L_{\phi}/\mu_{\phi}$.

We propose an alternative approach by applying Hyperfast second-order methods to minimize the function~$\Psi$. Since basic Hyperfast second-order methods are a special implementation of third-order method~\cite{nesterov2019implementable,gasnikov2019near,doikov2019contracting,nesterov2020superfast,nesterov2020inexact,kamzolov2020near}, each their iteration  requires to minimize the regularized third-order Taylor polynomial:
\begin{align}\label{eq:tensor}
    \min_{y\in\mathbb{R}^d}\bigg\{ \langle \nabla \Psi(x), y {-} x \rangle  +\frac{1}{2}\nabla^2  \Psi(x)[y - x]^2 +   \frac{1}{6}\nabla^3 \Psi(x) \left[y {-} x\right]^3 + \frac{L_{\Psi,3}}{8}\|y {-} x\|_2^4 \bigg\}.
\end{align}
It is shown in~\cite{nesterov2019implementable} that the objective in~\eqref{eq:tensor} is relatively smooth and strongly convex with respect to the function~$a(y) = \frac{1}{2}\nabla^2 \Psi(x)[y - x]^2 + \frac{L_{\Psi,3}}{8}\|y - x\|_2^4$ with~$\mu_{\Psi/a} = 1 - {1}/{\sqrt{2}}$,~$L_{\Psi/a} = 1 + {1}/{\sqrt{2}}$. Since~$\kappa_{\Psi/a}$ is a constant, the complexity of solving~\eqref{eq:tensor} is, up to logarithmic factors, the same as for minimizing~$a(y)$. 
In turn, the complexity of solving this problem, up to logarithmic factors, is the same as the complexity of a quadratic programming problem and can be estimated by the complexity of matrix inversion~\cite{nesterov2020inexactB}. 
To sum up, the arithmetic operations complexity of minimizing the function~$\Psi$ by the Restarted Hyperfast second-order method has the form
\begin{equation}\label{eq:compl}
    \widetilde{O}\left( \left({s^2n}
    + {d}^{\log_2 7} \right)\cdot {\left(\frac{L_{\phi,3}R^2}{\mu_{\phi}}\right)^{1/5}} \right),
\end{equation}
see {\cite{gasnikov2019near,nesterov2020inexact,kamzolov2020near}} for more details on arithmetic complexity of each iteration of the basic Hyperfast method. The first term in~\eqref{eq:compl}, i.e., {$s^2n$}, is due to the complexity of Hessian calculation. The second term, i.e.~${d}^{\log_2 7}$, corresponds to the complexity of Hessian inversion, e.g., by the matrix inversion lemma using Strassen's algorithm~\cite{huang2016strassen}. The term~${\left(\frac{L_{\phi,3}R^2}{\mu_{\phi}}\right)^{1/5}}$ comes from the estimate for the number of iterations of the basic Hyperfast second-order method, see Theorem \ref{Th:Hyper_V_compl}.
Additionally, we may expect~$R^2 = O(d)$, since~$\text{dim } \,x_* = d$ and~$L_{\phi,3} = O\big(\frac{1}{n}\sum_{i=1}^{n}\|\eta_i\xi_i\|_2^4\big) = O(s^2)$ since we consider sparse logistic regression.

Without loss of generality, we can assume that the parameter~$n$ can be set such that~${d}^{\log_2 7} = O\left({s^2n}\right)$. In this case, the Hyperfast second-order method with complexity \eqref{eq:compl} outperforms accelerated variance reduced algorithms with complexity \eqref{VR} if~$\mu_{\phi}\lesssim s^{-3}n^{-2}~$. Where~$\lesssim$, and~$\simeq$ mean the same as~$\leq$ and~$=$, but up to dimension-dependent factors of the order~$O(1)$.
For the particular case of sparse logistic regression problems, our focused application, we can assume that~$s=\widetilde{O}(1)$. Therefore, we have that if~$d \lesssim n^{0.356}$ and~$\mu_{\phi}\lesssim n^{-2}$, or, equivalently, if~$d^{\log_2 7}\lesssim n \lesssim \mu_{\phi}^{-1/2}$, then, the Hyperfast second-order method has smaller arithmetic operations complexity than variance reduced algorithms. 
The last inequality is reasonable when the requirement for the accuracy is high. Indeed, in practice, via regularization~\cite{gasnikov2017universal}, it is reasonable to set~$\mu_{\phi}\simeq \mu_F \simeq \varepsilon/R \simeq \varepsilon/d$, where~$\varepsilon > 0$ is a desired accuracy. 
Thus, in this case we can rewrite the last inequality as~$\varepsilon \lesssim n^{-1.644}$ ($d^{2.81}\lesssim n \lesssim \varepsilon^{-0.61}$). We can conclude that Hyperfast second-order methods are better when our goal is to solve sparse logistic regression with loss~\eqref{eq:log_reg} with high accuracy. This result can be strengthened by using parallelization. 
In the  complexity bound \eqref{VR} for variance reduced algorithms, only the first term can be improved by applying parallelization on~$n$ nodes.
On the contrary, in the bound \eqref{eq:compl} for Restarted Hyperfast method, the first term can be improved by parallelization on~$n$ nodes, and the second can be improved by parallelization on~$d$ nodes.

To conclude, high-order methods are competitive from the theoretical point of view for large-scale convex problems that require high accuracy of the solution, especially when the problem is sparse. 
Further improvements can potentially be  achieved by using inexact tensor methods~\cite{nesterov2020inexactB,doikov2020inexact,agafonov2020inexact,kamzolov2020optimal} to save some computational work.

\section{Numerical Analysis and Implementation Details}\label{sec:numerics}

In this section, we present numerical experiments and implementation details of Algorithm~\ref{algo:inspag}. Namely, on the example of regularized logistic regression, we demonstrate the practical performance of InSPAG method with Hyperfast subsolver (InSPAG+Hyperfast) and compare it with the state-of-the-art methods such as DANE, DANE-HB and SPAG with SDCA subsolver. For the logistic regression, we show that InSPAG+Hyperfast outperforms other methods even for huge-dimensional problems with 710M samples and 3.2M features.

We work with binary classification problems with regularized logistic regression cost function~\eqref{eq:log_reg} on a public datasets from LibSVM1\footnote{\url{https://www.csie.ntu.edu.tw/~cjlin/libsvmtools/datasets/binary.html}}, namely
RCV1~\cite{rcv1},  and a proprietary large-scale in-house dataset that was generated from the click logs of a large-scale commercial system for mobile app install ads. The main statistics of the datasets are shown in Table~\ref{tbl:stats}.

\begin{table}[h]
    \centering
    \begin{tabular}{c||c|c|c|c}
        Dataset &~$N$ &~$d$ & Feat. & Size\\\hline\hline
        RCV1 &20k & 47k& 74.05 & 13.7 \\\hline
        In-house &710M & 3,246k & 109.86 & 650.8k \\\hline
    \end{tabular}
    \caption{Statistics of the datasets.~$N$ is the number of samples,~$d$ is the number of features, Feat. is the average number of dense features, and Size is the data size in MB. }
    \label{tbl:stats}
\end{table}

We obtained an MPI-based distributed implementation of SPAG from the authors of~\cite{pmlr-v119-hendrikx20a} and modified it to run on an Apache Spark~\cite{apachespark} cluster. As shown in Algorithm~\ref{algo:inspag}, InSPAG switches between two phases: a parallel gradient computation phase and a central-node optimization phase in which we run the Hyperfast second-order method in Algorithm~\ref{alg:rest-Hyp-dist-V}. In our implementation, the driver carries the central-node optimization phase while executors compute the gradient. The code for the implementations was developed in PyTorch~\cite{pytorch}. Algorithm~\ref{alg:rest-Hyp-dist-V}, in each iteration of the basic Hyperfast method, requires a line-search where to calculate a test point the full step \eqref{eq:tensor} is made. The number of such line-search steps is theoretically bounded above by~$O(\log(\e^{-1}))$. However, we observe that the line-search ends in approximately~$5$ trials in practice. Therefore, we bound the number of iterations executed in the line-search procedure. Additionally, our experiments show that the number of steps required in the line-search procedure decreases as more iterations of Algorithm~\ref{algo:inspag} is executed. 
In the execution of the third-order step \eqref{eq:tensor} it is sufficient to approximate the product of the third derivative with two vectors.
To do this, we use off-the-shelf automatic differentiation codes and observe that the resulting computational complexity is equivalent approximately to~$4-6$ gradient computations.

As explained in Sect. 5.2 from \cite{nesterov2020inexact}, or Algorithm 2 from \cite{kamzolov2020hyperfast}, the problem \eqref{eq:tensor} is solved by Bregman proximal gradient method under relative smoothness and strong convexity assumption~\cite{lu2018relatively}.
Each step of this algorithm applied to \eqref{eq:tensor} requires to solve the problem
\begin{align}\label{eq:inner_most}
    \min_{s \in \R^d} \left\{ \la c,s\ra + \frac{1}{2}\la \nabla^2 \Psi(x) s,s\ra +\frac{L}{4} \|s\|^4_2\right\},
\end{align}
where the vector~$c$ involves~$\nabla \Psi(x)$ and~$\nabla^3\Psi(x)[s]^2$,~$L$ is some regularization parameter. We solve problem~\eqref{eq:inner_most} using ADAM~\cite{kingma2014adam} since then the gradient~$c+ \nabla^2 \Psi(x) s + L \|s\|^2_2 s$ of the objective uses the Hessian only through Hessian-vector products which can be calculated using automatic differentiation. We observed that in practice this takes approximately~$2-3$ times the time required for gradient computation. 
Thus, on the lowest level, our method is a first-order method with a Hessian-vector product and a third-order derivative product with two vectors computed by automatic differentiation techniques. 
The full Hessians or full third-order derivatives are not computed but are used for the method to exploit the additional curvature of the objective and improve the practical convergence speed. Moreover, the central node uses GPU to accelerate the various Hessian-related matrix-vector operations in the algorithm. We believe our implementation\footnote{\url{https://github.com/OPTAMI/OPTAMI/}} to be the first practical implementation of an algorithm from the family of Hyperfast or even a wider family of higher-order optimizers that can operate on data at the above dimensionality.

We compare Algorithm~\ref{algo:inspag} with the inner solver being Algorithm~\ref{alg:rest-Hyp-dist-V} and Algorithm~\ref{algo:inspag} with the inner solver being Stochastic Dual Coordinate Ascent (SDCA)~\cite{sdca} used in~\cite{pmlr-v119-hendrikx20a}. For the RCV1 dataset, we also compare the performance of  Algorithm~\ref{algo:inspag} versus DANE~\cite{shamir2014communication} with both SDCA and Hyperfast as the central-node solver. We used~$n=10^4$ samples for preconditioning,~$\lambda = 10^{-5}$,~$\sigma = 2 \times 10^{-5}$, constant~$L_{F / \phi}=0.01$, and a practical approximate~$10^{-2}$ for~$R_{\phi}^2$. We set the precision of the auxiliary subproblem to~$10^{-4}$. Other parameters:~$L_{3}=0.005$, the learning rate of ADAM is set to~$1$, and the number of iterations of ADAM is~$2$. Figures \ref{fig:rcv1} and \ref{fig:time} show results for the RCV1 dataset. The point~$\hat{x}$ is set as the point where the minimal cost was achieved overall the iterations and runs of the algorithm and serves as a proxy point used instead of the minimizer, which is in general unknown. We see that Algorithm~\ref{algo:inspag} outperforms DANE regardless of the subsolver used. Moreover, InSPAG-SDCA has better performance during initial iterations. However, InSPAG-Hyperfast outperforms all other methods by accuracy. Also, we find that Hyperfast iterations are faster than SDCA near the minimum point. For example, the first five iterations take about~$20$ seconds each, and the last five take about~$1.5$ seconds each. Hence, suggesting that some combination of methods would be used in practice. However, the Hyperfast approach finds better solutions overall.

\begin{figure}[t!]
    \centering
    \includegraphics[width=0.8\linewidth]{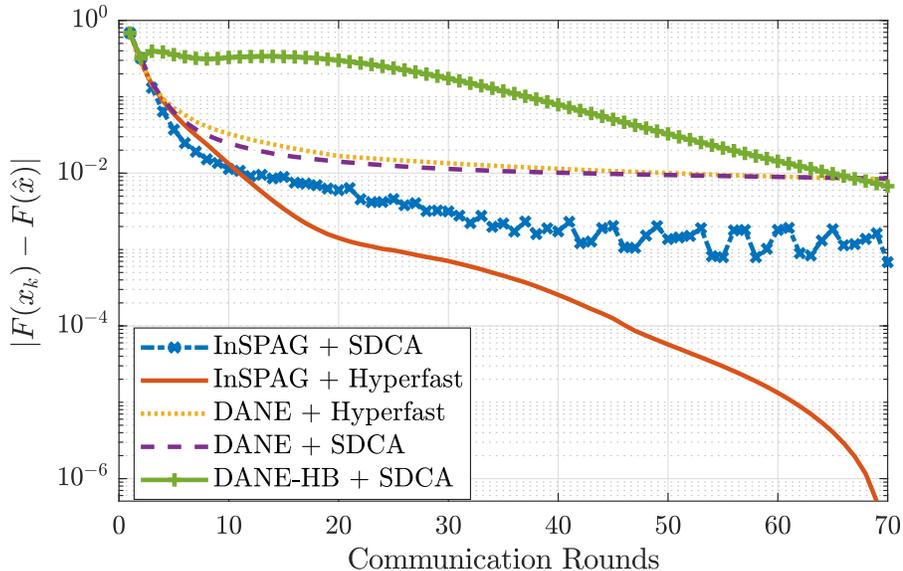}
    \vspace{-0.6cm}
    \caption{Comparison of the communication rounds number for the dataset RCV1.}
    \label{fig:rcv1}
\end{figure}

\begin{figure}[t!]
    \centering
    \includegraphics[width=0.8\linewidth]{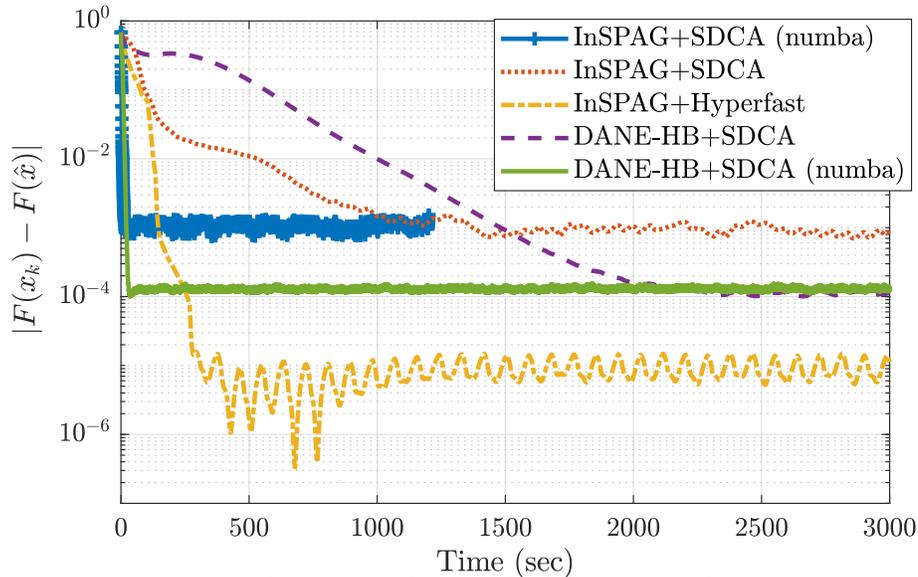}
    \vspace{-0.6cm}
    \caption{Wall clock time performance of the InSPAG method for the dataset RCV1. ``numba'' indicates implementation using \textit{Numba: A High Performance Python Compiler}.}
    \label{fig:time}
\end{figure}

\begin{figure}[t!]
    \centering
    \includegraphics[width=0.8\linewidth]{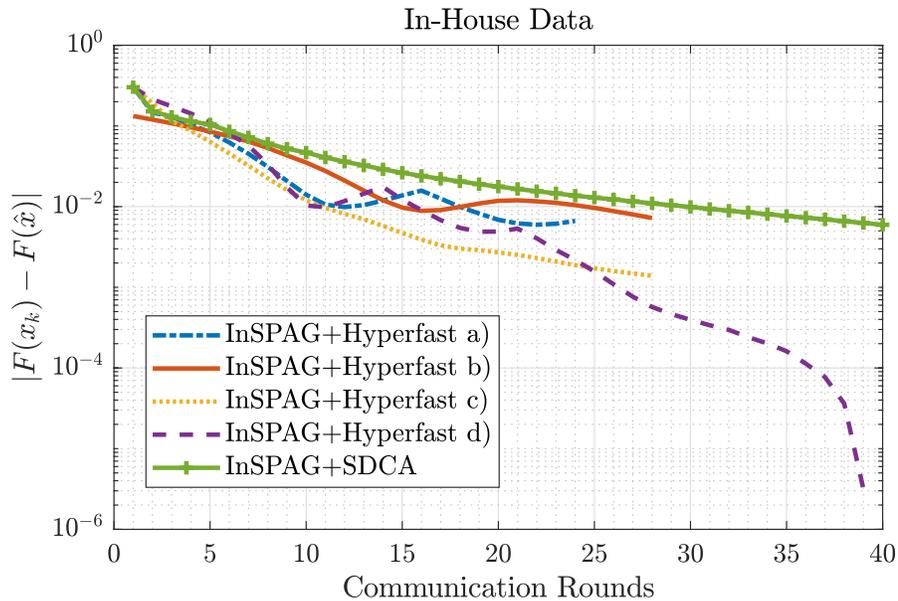}
    \vspace{-0.6cm}
    \caption{Comparison of the communication rounds number for the \text{in house} dataset. a)~$L_3=10$, ADAM learning rate~$0.01$,~$n=10000$; b)    ~$L_3=100$, ADAM learning rate~$0.1$,~$n=10000$; c)~$L_3=10$, ADAM learning rate~$0.1$,~$n=10000$; d)~$L_3=15$, ADAM learning rate~$0.01$,~$n=1000$.}
    \label{fig:in-house}
\end{figure}

\begin{figure}[t!]
    \centering
    \includegraphics[width=0.8\linewidth]{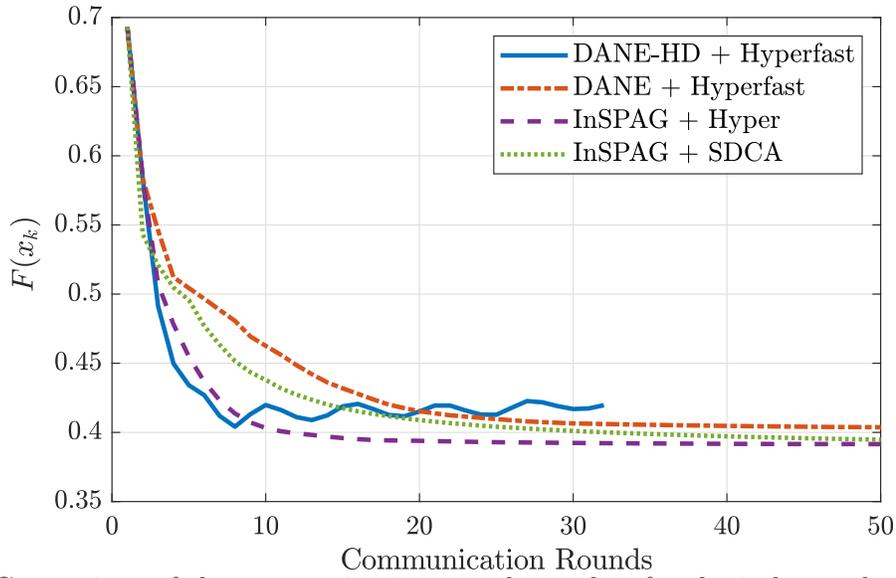}
    \vspace{-0.6cm}
    \caption{Comparison of the communication rounds number for the \text{in house} dataset for different methods. }
    \label{fig:in-house2}
\end{figure}

Figures \ref{fig:in-house}, \ref{fig:in-house2} show the results of the comparison on the in-house dataset (split over 200 nodes, i.e.,~$m=200$) with~$\lambda = 1\times 10^{-7}$,~$\sigma = 2  \times 10^{-5}$. Other parameters are described in Table~\ref{tbl:inhouse}. We see that InSPAG-Hyperfast outperforms InSPAG-SDCA for this large-scale dataset.  

\begin{table}[h]
    \centering
    \begin{tabular}{c||c|c|c|c}
        Run &~$L_3$ & ADAM &~$n$ &~$\mu$\\\hline\hline
        a) &~$10$ &~$0.01$ &~$1\times 10^4$&$2\times 10^{-5}$ \\
        b) &~$100$ &~$0.1$ &~$1\times 10^4$&$2\times 10^{-5}$ \\
        c) &~$10$ &~$0.1$ &~$1\times 10^4$&$2\times 10^{-5}$ \\
        d) &~$15$ &~$0.01$ &~$1\times 10^3$&$2\times 10^{-5}$ \\
    \end{tabular}
    \caption{Parameter selection for experiments on in-house data. }
    \label{tbl:inhouse}
\end{table}

Figure~\ref{fig:innner_times} shows the times required by the Hyperfast method in Algorithm~\ref{alg:rest-Hyp-dist-V} and the SDCA Method from~\cite{sdca} to complete their inner iterations at communication rounds~$0$,~$6$, and~$46$. The~$x$-axis is the iteration number, and the~$y$-axis is the time required by the corresponding algorithm to complete an inner iteration. We can observe that in the communication round~$0$, the cost time required by both methods is approximately the same on average. However, for communication rounds~$6$ and~$46$, the Hyperfast method outperforms SDCA, requiring less time to complete an iteration.

\begin{figure}[t!]
    \centering
    \includegraphics[width=1.1\linewidth]{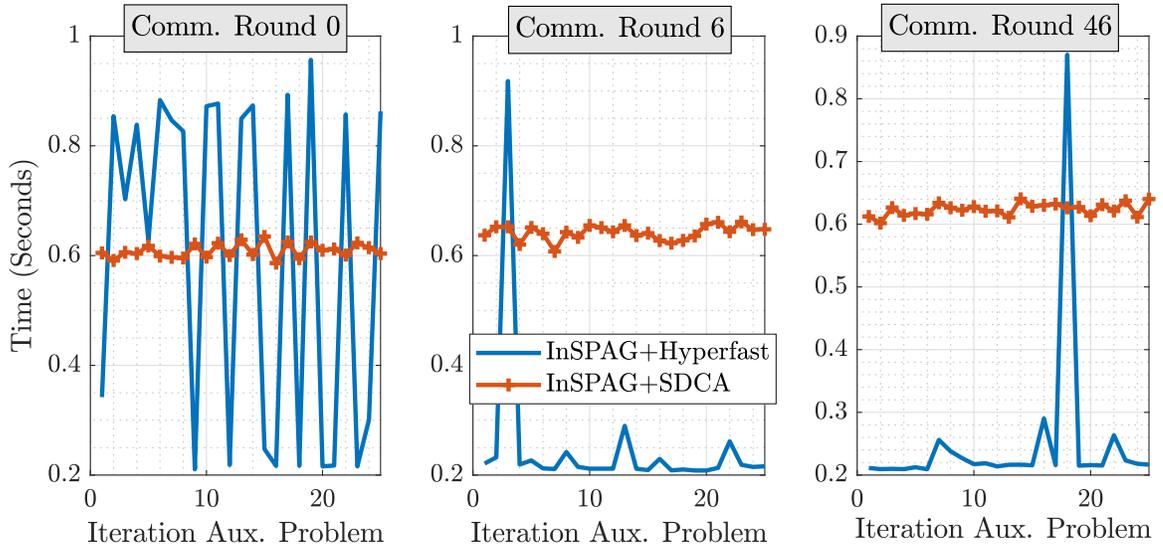}
    \vspace{-0.6cm}
    \caption{The time complexity per iteration for the Hyperfast method in Algorithm~\ref{alg:rest-Hyp-dist-V} and the SDCA Method from~\cite{sdca} at communication rounds~$0$,~$6$, and~$46$. The~$x$-axis is the iteration number, and the~$y$-axis is the the time required by the corresponding algorithm to complete its inner iteration.}
    \label{fig:innner_times}
\end{figure}

Figure~\ref{fig:loss_wall} on the left shows the loss function~$F(x_k)$ evaluated at the point ~$x_k$ generated by iteration~$k$ as a function of the wall clock time recorded by the InSPAG method in Algorithm~\ref{algo:inspag}. Markers identify when an iteration has been completed. In this case we used the Hyperfast method in Algorithm~\ref{alg:rest-Hyp-dist-V} as the inner solver. Moreover, we show the dependency on the number~$n$ of points used for preconditioning. We observe that for different values of~$n$, the final loss is about the same. However, as~$n$ increases, the wall clock time required increases as well.
On the other hand, the right figure shows the loss function~$F(x_k)$ evaluated at the point~$x_k$ generated by iteration~$k$ as a function of the number of communication rounds. As expected, when the number of data points used for preconditioning increases, the number of required communication rounds decreases. However, this implies that the central node needs to solve a bigger problem at every iteration and it takes longer to solve it.

\begin{figure}[t!]
    \centering
    \includegraphics[width=1.1\linewidth]{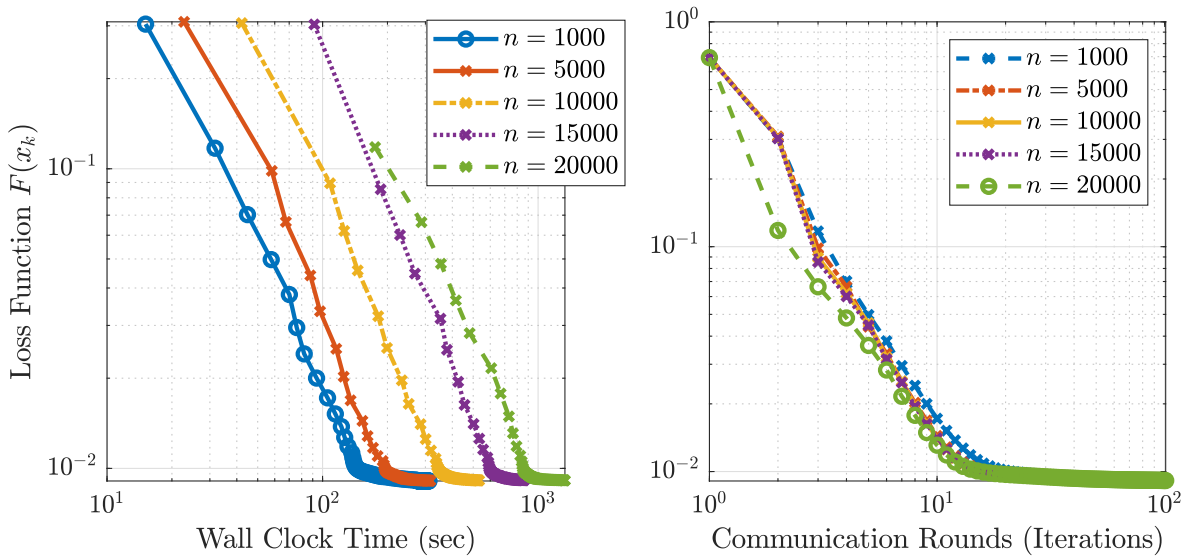}
    \vspace{-0.6cm}
    \caption{A comparison of the wall clock times and communication rounds for the InSPAG method in Algorithm~\ref{algo:inspag} for different number of data points used for preconditioning. On the left, the~$x$-axis indicates time in seconds, and on the right the~$x$-axis indicates number of communication rounds. In both cases the~$y$-axis is the loss function at the current iteration.}
    \label{fig:loss_wall}
\end{figure}

Figure~\ref{fig:cloc_wall} shows the wall clock time required by the central node to solve the auxiliary problem for every communication round. The~$x$-axis shows the number of communication rounds, and the~$y$-axis shows the clock time in seconds. Additionally, we show the results for different values of the preconditioning parameter~$n$. As~$n$ increases, the time required for the solution of the auxiliary problem increases as well. However, the time complexity of the auxiliary subproblem decreases as the number of communication rounds increases. 

\begin{figure}[t!]
    \centering
    \includegraphics[width=0.8\linewidth]{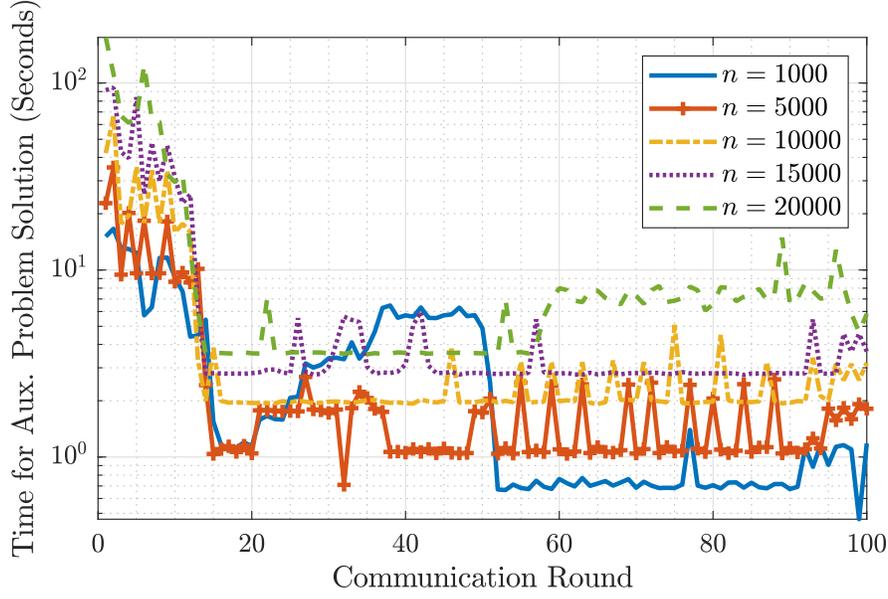}
    \caption{Time complexity for the solution of the auxiliary subproblem for different number of preconditioning data points. The~$x$-axis shows the number of communication rounds, and the~$y$-axis shows the clock time in seconds. }
    \label{fig:cloc_wall}
\end{figure}

\section{Hyperfast Second-Order Method for Uniformly Convex Functions}\label{app:uni}
For the sake of completeness, in this section we consider general problem
$
x_* = \arg\min_{x \in Q} f(x),
$
where~$Q$ is closed convex bounded set,~$f$ has~$L_3$-Lipschitz third-order derivative.
We also assume that the objective~$f(x)$ is uniformly convex of degree~$4 \geq q \geq 2$ on the convex bounded set~$Q$, i.e., there exists~$\sigma_q >0$ s.t.
\begin{equation}
    \label{eq:unif-conv-def}
    f(y) \geq f(x) + \la \nabla f(x) , y-x \ra + \frac{\sigma_q}{q}\|y-x\|_2^q,\quad \forall x,y \in Q.
\end{equation}
As a corollary,
\begin{equation}
    \label{eq:unif-conv-opt-point}
    f(y) \geq f(x_*) +  \frac{\sigma_q}{q}\|y-x_*\|_2^q, \forall y \in Q.
\end{equation}

\begin{theorem}[{\cite{nesterov2020inexact}[Theorem 2]}]
\label{Th:Hyperfast_conv_1}
Let sequence~$x^k$,~$k \geq 0$ be generated by Hyperfast Second-Order Method~\cite{nesterov2020inexact}[Eq.3.6] for~$p=3$ and~$\beta=1/2$ and with auxiliary steps described in~\cite{nesterov2020inexact}[Sect. 5.2]. Then 
\[
f(x_k) - f^* \leq \frac{3\cdot4^3L_3R_0^4}{1-\beta}\left[1+\frac{2(k-1)}{4} \right]^{-5} \leq \frac{3\cdot4^4L_3R_0^4}{16k^5}=\frac{\hat{c}L_3R_0^4}{k^5},
\]
where~$R_0$ is such that~$\|x_0-x^*\|_2\leq R_0$,~$\hat{c}=48$. 
\end{theorem}

We show how the restart technique can be used to accelerate Hyperfast second-order method under additional assumption of uniform convexity.

\begin{algorithm}
\caption{Restarted Hyperfast Second-Order Method}
\label{alg:rest-Hyp}
\begin{algorithmic}[1]
   \REQUIRE~$q$,~$\sigma_q$,~$z_0, \Delta_0$ s.t.~$f(z^0)-f^* \leq \Delta_0.$
   \FOR{$k=0,1,...$}
			\STATE
			$
			    \text{Set} \quad \Delta_k = \Delta_0\cdot2^{-k} \quad \text{and} \quad N_k = \max\left\{\left\lceil \left( \frac{2\hat{c} L_3 q^{\frac{4}{q}}}{\sigma_q^{\frac{4}{q}}} \Delta_k^{\frac{4-q}{q}} \right)^{\frac{1}{5}}\right\rceil,1 \right\}.
			$
			\STATE Set~$z_{k+1}=y_{N_k}$ as the output of the basic Hyperfast method started from~$z_k$ and run for~$N_k$ steps.
			\STATE Set~$k=k+1$.
			\ENDFOR		
		\ENSURE~$z_k$.
\end{algorithmic}
\end{algorithm}

\begin{theorem}
\label{Th:main-unif-conv-conv}
Let sequence~$z^k$,~$k \geq 0$ be generated by Algorithm \ref{alg:rest-Hyp}. Then 
\[
\frac{\sigma_q}{q}\|z_k-x_*\|_2^q \leq f(z_k) - f^* \leq \Delta_0 \cdot 2^{-k},
\]
and the total number of steps of the basic Hyperfast method is bounded by ($c$ is the constant in Theorem 1.)
\[
\left(2\hat{c} q^{\frac{4}{q}}\right)^{\frac{1}{5}} \frac{L_3^{\frac{1}{5}} }{\sigma_q^{\frac{4}{5q}}} (\Delta_0)^{\frac{4-q}{5q}} \cdot \sum_{i=0}^k  2^{-i\frac{4-q}{5q}} + k.
\]

\end{theorem}
\begin{proof}
Let us prove the first statement of the Theorem by induction. For~$k=0$ it holds. If it holds for some~$k\geq 0$, by choice of~$N_k$, we have that
\[
\frac{\hat{c}L_3}{N_k^{5}} \left( \frac{q \Delta_k}{\sigma_q}\right)^{\frac{4}{q}} \leq \frac{\Delta_k}{2}.
\]
By \eqref{eq:unif-conv-opt-point},
\[
\|z_k - x_*\|_2^{4} \leq \left( \frac{q (f(z_k)-f^*)}{\sigma_q}\right)^{\frac{4}{q}} \leq \left( \frac{q \Delta_k}{\sigma_q}\right)^{\frac{4}{q}}
\]
since, by our assumption,~$q\leq 4$.
Combining the above two inequalities and Theorem \ref{Th:Hyperfast_conv_1}, we obtain
\[
f(z_{k+1}) -f^* \leq \frac{\hat{c}L_3\|z_k-x_*\|_2^{4}}{N_k^{5}} \leq \frac{\Delta_k}{2} = \Delta_{k+1}.
\]
It remains to bound the total number of steps of the basic Hyperfast method. Denote~$\tilde{c} = \left(2\hat{c} q^{\frac{4}{q}}\right)^{\frac{1}{5}}$.
\[
\sum_{i=0}^k N_i \leq \tilde{c} \frac{L_3^{\frac{1}{5}} }{\sigma_q^{\frac{4}{5q}}} \sum_{i=0}^k (\Delta_0\cdot 2^{-i})^{\frac{4-q}{5q}} + k \leq \tilde{c} \frac{L_3^{\frac{1}{5}} }{\sigma_q^{\frac{4}{5q}}} (\Delta_0)^{\frac{4-q}{5q}} \cdot \sum_{i=0}^k  2^{-i\frac{4-q}{5q}} + k.
\]\end{proof}

Let us make a remark on the complexity of the restarted scheme in different settings. It is easy to see from Theorem \ref{Th:main-unif-conv-conv} that, to achieve an accuracy~$\e$, i.e.{,} to find a point~$\hat{x}$ s.t.~$f( \hat{x}) -f^* \leq \e$, the number of tensor steps in Algorithm \ref{alg:rest-Hyp} is
\[
O\left(\frac{L_3^{\frac{1}{5}} }{\sigma_q^{\frac{4}{5q}}} (\Delta_0)^{\frac{4-q}{5q}} + \log_2 \frac{\Delta_0}{\e}\right), q<4, \text{and} \;\;  O\left(\left(\left(\frac{L_3 }{\sigma_{4}}\right)^{\frac{1}{5}}+1 \right)\log_2 \frac{\Delta_0}{\e} \right), q=4.
\]

\section{Conclusions}\label{sec:conclusions}

We study the distributed optimization problem 
of minimizing empirical risk with smooth and (strongly) convex losses and i.i.d. data stored at nodes. 
Building upon the recent result on statistical preconditioning, we propose an algorithm that iteratively minimizes the objective function taking advantage of the statistical similarity of the cost functions across the nodes. Such statistical preconditioning requires solving an auxiliary optimization problem at a designated central node. Contrary to existing approaches, we analyze the case where this auxiliary problem is solved inexactly. Moreover, we provide the conditions on the accuracy of the solution that guarantees convergence at the same rate as the algorithm with access to exact minimizers of the auxiliary problem.
Additionally, we extend recently proposed Hyperfast second-order methods to the class of uniformly convex functions with bounded fourth-order derivatives. We show that the auxiliary problem in the statistically preconditioned distributed algorithm can be solved efficiently at a linear rate via this Hyperfast second-order method. We analyze the complexity of the proposed combination of the inexact statistically preconditioned algorithm with the Hyperfast second-order sub-solver and show that it converges linearly with the improved condition number. Finally, we show the first empirical results on implementing high-order methods on large-scale problems, where the dimension is of the order of~$3$ million, and the number of samples is~$700$ million. As a future research direction we indicate the application of the proposed algorithm to the regularized Wasserstein barycenter problem, which can be expressed as the minimization of large sum of higher-order smooth softmax functions~\cite{dvurechensky2018decentralize}.

\section*{Funding}
The work by D. Kamzolov was supported by a grant for research centers in the field of artificial intelligence, provided by the Analytical Center for the Government of the Russian Federation in accordance with the subsidy agreement (agreement identifier 000000D730321P5Q0002) and the agreement with the Moscow Institute of Physics and Technology dated November 1, 2021 No. 70-2021-00138.
The work by C. Uribe was supported by the Yahoo! Faculty Engagement Program and by the National Science Foundation under Grants No. 2211815 and No. 2213568. The work by P. Dvurechensky was funded by the Deutsche Forschungsgemeinschaft (DFG, German Research Foundation) under Germany's Excellence Strategy – The Berlin Mathematics Research Center MATH+ (EXC-2046/1, project ID: 390685689).

\section*{Acknowledgement}
The authors are grateful to Hadrien Hendrikx for sharing the code of the SPAG algorithm \url{http://proceedings.mlr.press/v119/hendrikx20a.html}.

\bibliography{bib/all_refs3,bib/PD_references,bib/references}

\end{document}